\documentclass[12pt,a4paper,oneside]{article}
\usepackage{amsfonts, amsmath, amssymb,latexsym, amsthm}
\usepackage{epsfig}
\usepackage{color}
\usepackage{enumerate}

\usepackage[cmtip,all]{xy}
\usepackage{array}
\usepackage{enumitem}
\usepackage{makecell}
\usepackage{multirow}

\usepackage{tikz} 
\usetikzlibrary{arrows,shapes}

\usepackage[latin1]{inputenc}

\newtheorem{theorem}{Theorem}[section]
\newtheorem{lemma}[theorem]{Lemma}
\newtheorem{proposition}[theorem]{Proposition}
\newtheorem{corollary}[theorem]{Corollary}

\theoremstyle{definition}

\newtheorem{example}[theorem]{Example}

\theoremstyle{remark}
\newtheorem{remark}[theorem]{Remark}

\newcommand\mc[1]{\multirow{2}{*}{#1}}

\numberwithin{equation}{section}

\usepackage[
colorlinks]{hyperref}

\usepackage{memhfixc} 
\usepackage
{hypcap} 
\hypersetup{
    bookmarksnumbered,
    pdfstartview={FitH},
    citecolor={red},
    linkcolor={red},
    urlcolor={red},
    pdfpagemode={UseOutlines}
}

\usepackage{fancyhdr}

\title{Classification of degree two curves in the symmetric square with positive self-intersection}
\author{Meritxell S\'aez}
\date{}
\begin{document}
\bibstyle{plain}

\maketitle

  \abstract{In this paper we give a precise classification of the pairs $(C,\widetilde{B})$ with $C$ a smooth curve of genus $g$ and $\widetilde{B}\subset C^{(2)}$ a curve of degree two and positive self-intersection. We prove that there are no such pairs if $g < p_a(\widetilde{B})< 2g-1$. We study the singularities and self-intersection of any degree two curve in $C^{(2)}$. Moreover, we give examples of curves with arithmetic genus in the Brill-Noether range and positive self-intersection on $C\times C$.}

{\parskip7pt
\noindent \textbf{Keywords}: Symmetric product, curve, irregular surface, curves in surfaces.

}

\maketitle

\section{Introduction}

In this paper we study pairs of irreducible curves $(C,\widetilde{B})$ with $g(C)\geq 2$ and $\widetilde{B}\subset C^{(2)}$ of degree two ($\widetilde{B}\cdot C_P=2$). We give a complete classification when $\widetilde{B}^2>0$ is satisfied. In particular, we prove that there is no such pair if $p_a(\widetilde{B})<2g(C)-1$.

A fundamental tool for this paper is the main result in \cite{MS2} where a characterization of curves in the symmetric square with a certain degree is given (Theorem \ref{charcurves}).  

In \cite[Question 8.6]{MPP2} the authors wonder if there exists a curve $B$ in a surface $S$ with $q(S)<p_a(B)<2q(S)-1$ (the Brill-Noether range) and $B^2>0$. This question relates with the existence of a curve $B$ of genus $q<p_a(B)<2q-1$ that generates an abelian variety of dimension $q$ (see \cite{Pi2}). We find a bound on the degree of such a curve $B$ lying in the symmetric square. For large $g(C)$ the bound suggests that such a curve should have low degree. This motivates the study of low degree curves in the symmetric square. In this paper we study the degree two case and in a future paper we will consider the degree three case (\cite{MS4}). We find that there is no curve with arithmetic genus in the Brill-Noether range, positive self-intersection and degree two in the symmetric square. When considering the preimages in $C\times C$ of some degree two curves we find examples of curves with arithmetic genus in the Brill-Noether range and positive self-intersection. Hence, answering the question in \cite{MPP2} as it has been done in \cite{PiAl}. We note that all such curves have arithmetic genus $2g(C)-2$, that is, the greatest possible genus in the Brill-Noether range.   

We prove first some preliminary results that have their own relevance, although they only play an auxiliary role in this paper.

Next, we study the geometry of degree two curves in the symmetric square with no regard to self-intersection or genus. In particular, we study in detail the possible singularities. We have two different cases depending on whether $\pi_C^*(\widetilde{B})$ is irreducible or not. If $\nu(\alpha)$ denotes the number of points in $C$ fixed by $\alpha \in {\rm Aut}(C)$, we show that for the reducible case the following holds.

\begin{proposition}\label{redtwo}
Let $\widetilde{B}\subset C^{(2)}$ be a degree two curve such that $\pi_C^*(\widetilde{B})$ is reducible. Then, there exists $\alpha\in {\rm Aut}(C)$ of degree at least three such that $\widetilde{B}=\{x+\alpha(x)\ |\ x\in C\}$.

Conversely, if $\alpha\in {\rm Aut}(C)$ has degree at least three, then $\widetilde{B}=\{x+\alpha(x)\ |\ x\in C\}\subset C^{(2)}$ with $\pi_C^*(\widetilde{B})$ being reducible.

Moreover, $\widetilde{B}$ has $\frac{1}{2}(\nu(\alpha^2)-\nu(\alpha))$ nodal singularities and normalization $C$. 
\end{proposition}

In the non-reducible case, using Theorem \ref{charcurves} and the fact that a degree two morphism is always Galois, we translate the study of pairs $(C, \widetilde{B})$ to the study of group actions on curves. Let $D_n$ denote the dihedral group of order $2n$. We prove the following.

\begin{theorem}\label{summarydihe}
Let $D$ be a curve, $i,j\in {\rm Aut}(D)$ two involutions with $\langle i,j \rangle =D_n$, $n\geq 3$. Let $C=D/\langle j \rangle$ and $B=D/\langle i \rangle$ be the quotient curves. Then, there is a degree one morphism $B\rightarrow C^{(2)}$ with image $\widetilde{B}$ such that $\widetilde{D}=\pi_C^*(\widetilde{B})$ is irreducible.

Conversely, if $\widetilde{B}\subset C^{(2)}$ is a degree two curve such that $\widetilde{D}=\pi_C^*(\widetilde{B})$ is irreducible and $B$ is the normalization of $\widetilde{B}$, then there exists a curve $D$ as above.

Moreover, $\widetilde{B}$ has $\frac{1}{4}(\nu((ij)^2)-\nu(ij))$ nodal singularities, $\widetilde{D}$ has $\frac{1}{2}\nu((ij)^2)$ nodal singularities and $D$ is the normalization of $\widetilde{D}$.
\end{theorem}

Then, we concentrate on degree two curves with positive self-intersection. We study all such curves getting the following Classification Theorem.

\begin{theorem}[Classification] \label{clasdeg2}  All pairs of curves $(C,\widetilde{B})$ with $C$ smooth, $\widetilde{B}\subset C^{(2)}$ such that $\widetilde{B}^2> 0$ and $\widetilde{B}\cdot C_P=2$ fall in one of the following cases:

\begin{enumerate}

\item[0.] $C$ is a curve of genus $2$ with the action of an automorphism of order $10$, $\alpha$, such that $\nu(\alpha)=1$, $\nu(\alpha^2)=3$, $\nu(\alpha^5)=6$, and $\widetilde{B}$ is the symmetrization of the graph of $\alpha$. There is a finite number of isomorphism classes of curves $C$ with such an automorphism. 

\end{enumerate}

\vskip 5pt

Let $B$ be the normalization of $\widetilde{B}$. In the following cases, $C=D/\langle j \rangle$ and $B=D/\langle i \rangle$ where $D$ is a smooth curve and $i,\ j\in {\rm Aut}(D)$ are two involutions such that $\langle i,j \rangle=D_n$ for $n\in\{10,6, 4\}$.  

\begin{enumerate}
\item For $\boldsymbol{n=10}$ there are three topological types of actions on $D$ that define three irreducible families with the following properties:

\vskip 4pt
{\setlength{\parindent}{-10pt} 
\small
\setlength{\tabcolsep}{4pt}

\setlength{\extrarowheight}{2pt}
\begin{tabular}{|c | c | c | c | c | c | c | c |}
\hline
\mc{$g(D)$}&\multirow{2}{*}{$g(C)$}&\multirow{2}{*}{$g(B)$}&\multirow{2}{*}{$\widetilde{B}^2$}& Moduli & Moduli&\multirow{2}{*}{Other properties}&\\
&&&&dim. of $D$&dim. of $C$& &\\
\hline
\mc{$5$} & \mc{$2$} & \mc{$3$} & \mc{$1$} & \mc{$1$} & \mc{$1$} & $D$ \small{hyperelliptic} & \mc{(D10.1)}\\
&&&&&& $\widetilde{B}$ \small{smooth}  &  \\
\hline
\mc{$4$} & \mc{$2$} & \mc{$2$} & \mc{$1$} & \mc{$1$} & \mc{$1$} & $D$ \small{hyperelliptic}& \mc{(D10.2)}\\
 &&&&&& $\widetilde{B}$ \small{has $1$ node} & \\
\hline
\mc{$6$} & \mc{$2$} & \mc{$3$} & \mc{$1$} & \mc{$2$} & \mc{$1$ or $2$} & $D$ \small{bielliptic} & \mc{(D10.3)}\\ 
&&&&&& $\widetilde{B}$ \small{smooth}  &  \\
\hline
\end{tabular}
}

\vskip 4pt

Furthermore, in all three families the curve  $B$ is hyperelliptic and $p_a(\widetilde{B})=2g(C)-1$.

\item For $\boldsymbol{n=6}$ there are ten topological types of actions on $D$ that define ten irreducible families.

One family with the following characteristics:

\vskip 4pt
{\setlength{\parindent}{-10pt} 
\small
\setlength{\tabcolsep}{4pt}

\setlength{\extrarowheight}{2pt}
\begin{tabular}{|c | c | c | c | c | c | c | c |}
\hline
\mc{$g(D)$} & \mc{$g(C)$} & \mc{$g(B)$} & \mc{$\widetilde{B}^2$} & Moduli & Moduli & \mc{Other properties} \\
&&&& dim. of $D$ & dim. of $C$ & & \\
\hline
$5$ & $2$ & $3$ & $2$ & $2$ & $2$ & $\widetilde{B}$ \small{has $1$ node} & (D6.1)\\
\hline
\end{tabular}
}

\vskip 4pt

Furthermore, in this family the curves $D$ and $B$ are hyperelliptic and $p_a(\widetilde{B})=2g(C)$. 

The other nine families have the following characteristics:

\vskip 4pt
{\setlength{\parindent}{-10pt} 
\small
\setlength{\tabcolsep}{4pt}
\setlength{\extrarowheight}{4pt}
\begin{tabular}{|c | c | c | c | c | c | c | c |}
\hline
\mc{$g(D)$} & \mc{$g(C)$} & \mc{$g(B)$} & \mc{$\widetilde{B}^2$} & Moduli & Moduli & \mc{Other properties} & \\ 
&&&& dim. of $D$ &  dim. of $C$ &  & \\
\hline
$7$ & $3$ & $4$ & $1$ & $2$ & $2$ & $\widetilde{B}$ has $1$ node  & (D6.2)\\[1pt]
\hline
$9$ & $3$ & $5$ & $1$ & $3$ & $3$ & $\widetilde{B}$ smooth  & (D6.3)\\[1pt]
\hline
$5$ & $2$ & $3$ & $2$ & $1$ & $1$ & $\widetilde{B}$ smooth  & (D6.4)\\[1pt]
\hline
$6$ & $3$ & $3$ & $1$ & $2$ & $2$ & $\widetilde{B}$ has $2$ nodes  & (D6.5)\\[1pt]
\hline
$8$ & $3$ & $4$ & $1$ & $3$ & $3$ & $\widetilde{B}$ has $1$ node  & (D6.6)\\[1pt]
\hline
$4$ & $2$ & $2$ & $2$ & $1$ & $1$ &  $\widetilde{B}$ has $1$ node & (D6.7)\\[1pt]
\hline
$6$ & $2$ & $3$ & $2$ & $2$ & $2$ & $\widetilde{B}$ smooth  & (D6.8)\\[1pt]
\hline
$5$ & $2$ & $2$ & $2$ & $2$ & $2$ &  $\widetilde{B}$ has $1$ node & (D6.9)\\[1pt]
\hline
$7$ & $2$ & $3$ & $2$ & $3$ & $2$ &  $\widetilde{B}$ smooth  & (D6.10)\\[1pt]
\hline
\end{tabular}
}
\vskip 4pt

Furthermore, $B$ and $C$ are bielliptic and $p_a(\widetilde{B})=2g(C)-1$.

\item For $\boldsymbol{n=4}$  there are three irreducible families of topological types of actions on $D$ with the following properties:

\vskip 4pt
{\setlength{\parindent}{-9pt} \setlength{\extrarowheight}{7pt}  \small

\begin{tabular}{|c | c | c | c | c | c | c | c |}
\hline
\mc{$g(D)$} & \mc{$g(C)$} & \mc{$g(B)$} & \mc{$\widetilde{B}^2$} & Moduli & \mc{Other properties} & \\ 
&&&& dim. of $D$ & & \\
\hline
\mc{$-1+s+\frac{1}{2}k$} & \mc{$\frac{s+k}{2}$} & \mc{$\frac{2s+k}{4}$} & \mc{$4$} & \mc{$\frac{2s+k-4}{4}$} & 
$\widetilde{B}$ has $\frac{k}{4}$ nodes & \mc{(D4.1)}
\\ 
&&&&&$s+k\geq 8$ &  \\[2pt]
\hline
\mc{$-2+s+\frac{1}{2}k$} & \mc{$\frac{s+k-2}{2}$} & \mc{$\frac{2s+k-4}{4}$} & \mc{$4$}  & \mc{$\frac{2s+k-4}{4}$} & 
$\widetilde{B}$ has $\frac{k}{4}$ nodes & \mc{(D4.2)}\\ 
&&&&& $s+k\geq 10$ & \\[2pt]
\hline
\mc{$-3+s+\frac{1}{2}k$} & \mc{$\frac{s+k-4}{2}$} & \mc{$\frac{2s+k-8}{4}$} & \mc{$4$}  & \mc{$\frac{2s+k-4}{4}$} &
$\widetilde{B}$ has $\frac{k}{4}$ nodes & (D4.3)\\ 
&&&&& $s+k\geq 12$ & \\[2pt]
\hline
\end{tabular}
}

\vskip 4pt
In all three families the curve $C$ is hyperelliptic, with any possible genus, and $p_a(\widetilde{B})=2g(C)$. Furthermore, $h^0(C^{(2)}, \mathcal{O}_{C^{(2)}}(\widetilde{B}))\geq 2$ and $\widetilde{B}$ is linearly equivalent to the sum of two coordinate curves.
\end{enumerate}
\end{theorem}

The moduli dimension of $C$ is not computed for $n=4$ because the freedom in the parameters $s$ and $k$ makes the computation not manageable with the techniques at hand.

Since actions of groups on algebraic curves is one of the main tools for the classification of degree two curves in $C^{(2)}$, in particular when the quotient curve is $\mathbb{P}^1$, some of our results can be linked with the study of the moduli dimension of a Nielsen class (see \cite{Fri}). We compute the moduli dimension for the families that appear in Theorem \ref{clasdeg2} for $n\in\{6,10\}$.

The outline of the paper is as follows. In Section \ref{selfintsect} we compute the self-intersection of a curve $\widetilde{B}\subset C^{(2)}$ of degree $d$ as in Theorem \ref{charcurves}. 

In Section \ref{groupact} we introduce a method to construct curves in a symmetric square using the action of a finite group on a curve and Theorem \ref{charcurves}.
In the final part of the section we give some background on the action of groups on curves. 

In Section \ref{dihe} we study curves of degree two in $C^{(2)}$. We consider first those with reducible preimage in $C\times C$ (Theorem \ref{redtwo}) and next those with irreducible preimage (Theorem \ref{summarydihe}).   

In Section \ref{proofthm} we prove Theorem \ref{clasdeg2}. First, we study the possible curves with reducible preimage and positive self-intersection. Next, we consider those with irreducible preimage. We study the numerical conditions determined by our hypothesis on the action of the dihedral group in a curve $D$. We define, when possible, a generating vector of the corresponding dihedral group acting on a curve $D$. For each generating vector we study (when possible) the moduli dimension of the curves $C$ that appear in this way. 

In Section \ref{sec:furthcomm} we study the conditions that a curve in $C^{(2)}$ of degree $d$ must satisfy to have positive self-intersection and arithmetic genus in the Brill-Noether range.  As a corollary of Theorem \ref{clasdeg2} we obtain that there are no such curves with degree $2$. Finally, we give examples of such curves in $C\times C$. 

\textbf{Notation:} We work over the complex numbers. By curve we mean a complex projective reduced algebraic curve. Let $C$ be a smooth curve of genus $g\geq 2$, we put $C^{(2)}$ for its $2$nd symmetric product. We denote by $\pi_C:C \times C \rightarrow C^{(2)}$ the natural map, and $C_P \subset C^{(2)}$ a coordinate curve with base point $P\in C$. We put $\Delta_C$ for the main diagonal in $C^{(2)}$, and $\Delta_{C \times C}$ for the diagonal of the Cartesian product $C\times C$. We denote by $p_a(C) = h^1(C,\mathcal{O}_C)$ the arithmetic genus and when $C$ is smooth by $g(C)= h^0(C, \omega_{C})$ the geometric genus (or topological genus). We will call node an ordinary singularity of order two.

For $\alpha \in \mathrm{Aut}(C)$, we denote by $o(\alpha)$ its order and by $\nu(\alpha)$ the number of points fixed by it. We put $\Gamma_{\alpha}$ for the curve in $C\times C$ given by the graph of $\alpha$, that is, $\Gamma_{\alpha}=\{ (x, \alpha(x)),\ x\in C\}$.

A compact Riemann surface $C$ will be called $\gamma$-hyperelliptic if there is a compact Riemann surface $\widetilde{C}$ of genus $\gamma$ and a holomorphic mapping  $p:C\rightarrow \widetilde{C}$ of degree two.

\section{Self-intersection of $\widetilde{B}$ in $C^{(2)}$}\label{selfintsect}

In this section we compute the self-intersection of a curve $\widetilde{B}\subset C^{(2)}$ of degree $d$ defined by a diagram of curves as in the following theorem.

\begin{theorem}[\cite{MS2}]\label{charcurves}
Let $B$ be an irreducible smooth curve such that there is no non-trivial morphism $B\rightarrow C$. Giving a degree one morphism from $B$ to the surface $C^{(2)}$ with image $\widetilde{B}$ of degree $d$ is equivalent to giving a smooth irreducible curve $D$ and a diagram 
\[
\xymatrix{
D \ar[d]_{(d:1)}\ar[r]^{(2:1)}& B\\
C &
}
\]
which does not reduce.
\end{theorem}

When $d$ is prime, that the diagram does not reduce is equivalent to the property that the diagram does not complete. That is, that does not exist a curve $H$ and maps such that we obtain a commutative diagram
\[
\xymatrix{
D \ar[d]_{(d:1)}\ar[r]^{(2:1)}& B \ar@{.>}[d]^{(d:1)}\\
C \ar@{.>}[r]_{(2:1)}& H.
}
\]

\begin{lemma}\label{selfint}
Let $\widetilde{B}\subset C^{(2)}$ be a curve given by a diagram which does not reduce 
\[
\xymatrix{
D \ar[d]_{(d:1)}^{g}\ar[r]^{(2:1)}_{f}& B\\
C &,
} 
\]
and let $\widetilde{D}=\pi_C^*(\widetilde{B})$. Then,
\[
\widetilde{B}^2= p_a(\widetilde{D})-1-d(2g(C)-2).
\]
\end{lemma}

\begin{proof}
From the adjunction formula for $\widetilde{B}\subset C^{(2)}$ together with the numerical equivalence $K_{C^{(2)}}\equiv_{num} (2g(C)-2)C_P-\frac{\Delta_{C}}{2}$ we deduce that
\begin{equation}\label{autB}
\widetilde{B}^2=2p_a(\widetilde{B})-2-d(2g(C)-2)+\widetilde{B}\cdot \frac{\Delta_C}{2}.
\end{equation}

Since $\widetilde{D}$ and $\widetilde{B}$ lay in two smooth surfaces, by adjunction, their canonical divisors are locally free, and thus they are Gorenstein curves. Then, by the Riemann-Hurwitz formula for the morphism $\widetilde{D}\rightarrow \widetilde{B}$ induced by $\pi_C$ we get 
\[
2p_a(\widetilde{D})-2=2(2p_a(\widetilde{B})-2)+\widetilde{B}\cdot \Delta_C,
\]
and therefore,
\begin{equation}\label{adjDB}
\widetilde{B}\cdot \frac{\Delta_C}{2}=p_a(\widetilde{D})-2p_a(\widetilde{B})+1.
\end{equation}

From \eqref{autB} and \eqref{adjDB} we get the expression in the Lemma \ref{selfint}.
\end{proof}

\section{Background on group actions}\label{groupact}

We consider now diagrams of curves such that both morphisms are Galois. We give a method to find diagrams that do not complete or decide if a given diagram completes, depending on the order of the group generated by the two automorphisms defining the diagram.

\begin{proposition}\label{prop}
Let $D$ be a projective smooth irreducible curve with the action of a finite group $G$. Let $\alpha, \beta \in G$ with orders $o(\alpha)=d\geq 2$ and $o(\beta)=e\geq 2$. Consider the diagram 
\[
\xymatrix{
D \ar^(.35){(e:1)}[r] \ar[d]^(.35){(d:1)} & D/{\langle\beta\rangle}=B\\
D/{\langle\alpha\rangle}=C &.
}
\]
Then,
\begin{enumerate}[label=(\arabic*)]

\item If the order of $\langle \alpha, \beta\rangle $ equals $e\cdot d$ then the diagram completes. \label{productequal}
\item If the order of $\langle \alpha, \beta\rangle $ is strictly greater than $e\cdot d$ then the diagram does not complete. \label{productsmaller}
\end{enumerate}
\end{proposition}

\begin{proof}
First we prove \ref{productequal}. Assume that $|\langle \alpha, \beta \rangle|=e\cdot d$. Let $F$ be the quotient of $D$ by the action of $\langle \alpha, \beta\rangle $. 
Then, the diagram completes. Indeed, one can define morphisms from $B$ and $C$ to $F$, using that both $B$ and $C$ are quotients of $D$ by subgroups of $\langle \alpha, \beta \rangle$.

Now, we prove \ref{productsmaller}. Assume that $|\langle \alpha, \beta \rangle|>e\cdot d$. By contradiction we assume that the diagram completes. That is, there exists a curve $H$ giving a commutative diagram
\[
\xymatrix{
D \ar[r]^{(e:1)} \ar[d]_{(d:1)} \ar[dr]|{(ed:1)}& B \ar[d]^{(d:1)}\\
C \ar[r]_{(e:1)}& H.
}
\]
The automorphisms $\alpha$ and $\beta$, and hence the group $\langle \alpha, \beta\rangle $, act on the fibers of $D \xrightarrow{(ed:1)} H$. Therefore, the orbit of a general point of $D$ by the action of $\langle \alpha,\beta\rangle $ must be contained in a fiber of $D \xrightarrow{(ed:1)} H$, but the cardinality of the first is strictly greater than the degree of the second, so this inclusion is not possible, and consequently such a curve $H$ does not exist. \qedhere
\end{proof}

\begin{remark}
If the order of $\langle \alpha, \beta\rangle $ is strictly less than $e\cdot d$ the diagram defined by these two automorphisms may or may not complete. For instance, if $\beta=\alpha^k$, then there is a natural morphism from $B$ to $C$,
\[
\xymatrix{
D \ar[r]^(.35){(e:1)} \ar[d]^(.35){(d:1)} & D/{\langle\alpha^k\rangle}=B \ar[dl]^(.4){(d/e:1)} \\
D/{\langle\alpha\rangle}=C &.
}
\]
In that case, the diagram completes if and only if $C$ covers a curve $H$ with degree $e$.
\end{remark}

To study the geometry of the curves defined by such diagrams we need to recall some basic facts about group actions on curves.

Let $C$ be a curve and let $G\subset \mathrm{Aut}(C)$ be a finite subgroup. For $P\in C$, set
\[
G_P=\{g\in G\ |\ g(P)=P\}
\]
the \textbf{stabilizer} of $P$.

\begin{proposition}\label{cyc}(\cite[III.7.7]{Far})
Assume $g(C)\geq 2$. Then, $G_P$ is a cyclic subgroup of $\mathrm{Aut}(C)$.
\end{proposition}

Given $\alpha\in {\rm Aut(C)}$, its graph $\Gamma_{\alpha}$ lies in $C\times C$ and is isomorphic to $C$. With a local computation one can see the following.

\begin{proposition}\label{diagtrans}
The diagonal in $C\times C$ cuts the graph of an automorphism transversally.
\end{proposition}

\begin{corollary}\label{grtrans}
Let $\alpha$ and $\beta$ be two automorphisms of a curve $C$. If $\alpha^{-1}\beta\neq 1$, then the graphs of $\alpha$ and $\beta$ in $C \times C$ intersect transversally and  $\Gamma_{\alpha}\cdot \Gamma_{\beta}=\nu(\alpha^{-1}\beta)$, the number of fixed points of the automorphism $\alpha^{-1}\beta$.
\end{corollary}

\begin{proof}
We transform the two considered graphs by the action of $1\times \alpha^{-1}$: 
\[
\begin{array}{lcl}
\Gamma_{\alpha}=\{(x, \alpha(x))\} &\xrightarrow{1\times \alpha^{-1}}& \{(x,x) \}=\Delta_{C \times C}\\
\Gamma_{\beta}=\{(x, \beta(x))\}&\xrightarrow{1\times \alpha^{-1}} &\{(x, \alpha^{-1}\beta (x))\}=\Gamma_{\alpha^{-1}\beta}.
\end{array} 
\]

Since the diagonal intersects transversally the graph of any automorphism, we deduce that the two graphs intersect also transversally in $\nu(\alpha^{-1}\beta)$ points. 
\end{proof}

\begin{lemma}\label{orbit}
Let $G$ be a finite group of order $n$ acting on a curve $C$. Given a point $P\in C$, let $\alpha$ be a generator of $G_P$. Then, we have that
\[
n=|G_P|\cdot |\{\textrm{conjugates of }G_P\}|\cdot |\{\textrm{points fixed by }\alpha\textrm{ in } \mathcal{O}_G(P)\}|.
\]
\end{lemma}

\begin{proof}
By Lagrange Theorem (\cite[Theorem 2.27]{Ros}) and the orbit-stabilizer Theorem (\cite[Theorem 5.7]{Ros}) we obtain that
\[
n=|G|=|G_P|\cdot [G:G_P]=|G_P|\cdot |\mathcal{O}_G(P)|.
\]

Since the point $P$ has stabilizer $G_P$, given a conjugate of $G_P$, we see that $\alpha G_P \alpha^{-1}=G_{\alpha(P)}$, the stabilizer of $\alpha(P)$. Moreover, given any element $\beta \in G$, $\beta(P)$ has stabilizer $G_P$ or one of its conjugates. Therefore, in the orbit of $P$ there are the same number of points with stabilizer each conjugate of $G_P$, and all conjugates of $G_P$ are stabilizers of points in the orbit. Hence,
\[
|\mathcal{O}_G(P)|= |\{\textrm{conjugates of }G_P\}|\cdot |\{\textrm{points fixed by }\alpha\textrm{ in } \mathcal{O}_G(P)\}|. \qedhere
\]
\end{proof}

Next, we recall the form of the Riemann-Hurwitz formula for group actions. Let $\mathrm{Br}$ be the branch locus of $f:C\rightarrow C/G$. Then, the Riemann-Hurwitz formula for $f$ reads
\begin{equation}\label{RHG}
2g-2=|G|(2g'-2)+|G|\sum\limits_{P\in \mathrm{Br}}\left(1-\frac{1}{m_P}\right)
\end{equation}
where $g$ and $g'$ are the genus of $C$ and $C/G$ respectively, and $m_P=|G_Q|$ with $f(Q)=P$. Since $f$ is Galois, it is totally ramified, and we call $m_P$ the order of the branch point $P$.

\begin{theorem}[Riemann's Existence Theorem]\label{rexthm}
The group $G$ acts on a curve of genus $g$, with branching type $(g'; m_1, \dots, m_r)$ if and only if the Riemann-Hurwitz formula is satisfied and $G$ has a $(g'; m_1, \dots, m_r)$ generating vector.
\end{theorem}

Where a $(g'; m_1, \dots, m_r)$ generating vector (or $G$-Hurwitz vector) is a $2g'+r$-tuple 
\[
(a_1, b_1, \dots, a_{g'},b_{g'}; c_1, \dots, c_r)
\] 
of elements of $G$ generating the group and such that $o(c_i)=m_i$ and 
\[
\prod\limits_{j=1}^{g'} [a_i,b_i]\prod\limits_{i=1}^{r}c_i=1.
\]
We call this last condition the \textbf{product one condition}. 

We remark that Riemann's Existence theorem is not a constructive result. It states the existence of such a curve, but it gives no further information about it. With the following theorem we will be able to compute the number of fixed points of each element $\gamma \in G$ acting on the curve.

\begin{theorem}(\cite{MB})
Let $C$ be a compact Riemann surface and $G$ a group of its automorphisms. Let $(a_1, b_1, \dots, a_{g'},b_{g'}; c_1, \dots, c_r)$ be a $(g';m_1, \dots, m_r)$-generating vector of $G$ describing the action of $G$ on $C$. For $1\neq \gamma \in G$ let $\varepsilon_i(\gamma)$ be $1$ or $0$ according as $\gamma$ is or is not conjugate to a power of $c_i$. 

Then the number $\nu(\gamma)$ of points of $C$ fixed by $\gamma$ is given by the formula
\[
\nu(\gamma)=|N_G(\langle \gamma \rangle)| \sum\limits_{i=1}^r \frac{\varepsilon_i(\gamma)}{m_i},
\]
where $N_G(\langle \gamma \rangle)$ is the normalizer of $\langle \gamma \rangle$ in $G$.
\end{theorem}

\section{Degree two curves}\label{dihe}

Now, we study curves of degree two in the symmetric square of a curve. First of all we observe that by the Hodge index theorem, an irreducible curve $\widetilde{B}$ of degree two in $C^{(2)}$ satisfies the inequality $\widetilde{B}^2\leq 4$. Moreover, when $\widetilde{B}^2= 4$ the curve is algebraically equivalent to twice a coordinate curve.

We present a lemma that will be useful in the discussion that follows. The proof uses basic group theory and the particular group structure of the dihedral groups and is left to the reader.

\begin{lemma}\label{cycdihe} Let $i$ and $j$ be two involutions generating a dihedral group $D_n$, $n\geq3$. Then, there is no cyclic subgroup containing $(ij)^2$ and one of the involutions $i$ or $j$.
\end{lemma}

We start by studying irreducible degree two curves $\widetilde{B}\subset C^{(2)}$ such that $\pi_C^*(\widetilde{B})$ is reducible. 

\begin{proof}[Proof of Proposition \ref{redtwo}]

Since there is a degree two morphism from $\pi_C^*(\widetilde{B})$ to $C$, if $\pi_C^*(\widetilde{B})$ reduces, then it consists of two copies of $C$, and the projections onto each factor of $C\times C$ are hence isomorphisms. This gives an automorphism of $C$, $\alpha$, as follows: take one component of $\pi_C^*(\widetilde{B})$, for each point $x$ in this component we define $\alpha(\pi_1(x))=\pi_2(x)$ where $\pi_1$, $\pi_2$ are the projections on the two factors of $C\times C$. Notice that the order of $\alpha$ must be at least $3$, because otherwise $\pi_C^*(\widetilde{B})$ would have only one component. 

Hence, we have $\pi_C^{*}(\widetilde{B})=C_1+C_2\subset C \times C$ with $C_1=\Gamma_{\alpha}$, $C_2=\Gamma_{\alpha^{-1}}$ and $\widetilde{B}=\{x+\alpha(x),\ x\in C\}$. The curve $\widetilde{B}$ has normalization $C$ and moreover the following holds:

First, $\boldsymbol{\widetilde{B}\cdot \Delta_C=2\nu(\alpha)}$. Indeed, consider
\[
\widetilde{B}\cap \Delta_C=\{x+\alpha(x)\ |\ \alpha(x)=x\}.
\]

The preimages of these points by $\pi_C$ correspond to points were $C_1$ and $C_2$ meet (transversally by Corollary \ref{grtrans}) over the diagonal. They intersect the diagonal transversally (by Proposition \ref{diagtrans}), and taking local coordinates we see that $\widetilde{B}$ and $\Delta_C$ are tangent at $x+x$ for $x$ a point fixed by $\alpha$.  

Second, $\boldsymbol{|\mathrm{Sing}\,\widetilde{B}|=\frac{1}{2}(\nu(\alpha^2)-\nu(\alpha))}$. Indeed, a general curve $C_P$ intersects $\widetilde{B}$ in two different points: $P+\alpha(P)$ and $P+\alpha^{-1}(P)$. Since $C_P\cdot \widetilde{B}=2$, when these two points are different they are smooth points on $\widetilde{B}$. To determine the singularities of $\widetilde{B}$ we need to study when these two points coincide. We have two possibilities:

Either $\alpha(P)=P$ and hence $\widetilde{B}$ intersects the diagonal in a smooth (tangent) point as we have just seen.

Or $\alpha(P)=\alpha^{-1}(P)\neq P$, that is, $P$ is fixed by $\alpha^2$ and not by $\alpha$. We observe that if $P$ is fixed by $\alpha^2$, then the point $\alpha(P)$ is also fixed by $\alpha^2$, and both give the same singularity $P+\alpha(P)=\alpha(P)+\alpha^2(P)$.

Finally, all singularities of $\widetilde{B}$ are \textbf{nodes}. Indeed, consider the normalization morphism 
\[
\begin{array}{rcl}
C & \longrightarrow & \widetilde{B}\subset C^{(2)}\\
x & \mapsto & x+\alpha(x).
\end{array}
\]
A singular point $x+\alpha(x)$ with $\alpha^2(x)=x\neq \alpha(x)$ has two preimages by the normalization morphism: $x$ and $\alpha(x)$, and hence $\widetilde{B}$ has two branches at $x+\alpha(x)$. Since $C_P\cdot \widetilde{B}=2$ the singularities have order two. Moreover, since the preimage of $\widetilde{B}$ by $\pi_C$ is formed by the graphs of $\alpha$ and $\alpha^{-1}$, which are transversal by Corollary \ref{grtrans}, and $\pi_C$ is a local isomorphism around these points we conclude that these singularities are nodes.  
\end{proof}

Next, we consider $\widetilde{B}\subset C^{(2)}$, with normalization $B$ and irreducible preimage by $\pi_C$. Let $D$ be the normalization of $\widetilde{D}:=\pi_{C}^*(\widetilde{B})$.

\begin{proof}[Proof of Theorem \ref{summarydihe}]

Regarding Theorem \ref{charcurves}, there exists a diagram which does not complete
\[
\xymatrix{
D \ar[r]^{(2:1)}_{f} \ar[d]_{(2:1)}^{g} & B \\
C &
}
\] 
defined by $\widetilde{B}\subset C^{(2)}$. Since both morphisms are of degree two, in $D$ there are two involutions $i$ and $j$ (the changes of sheet) such that $C=D/\langle j\rangle$ and $B=D/\langle i\rangle$. By Proposition \ref{prop} $i$ and $j$ generate a group of order at least five, that is, a dihedral group (see \cite{Mi1}).

We study now the singularities of $\widetilde{D}$ and $\widetilde{B}$.

Consider the morphism $g\times g:D \times D \rightarrow C \times C$. It is Galois with group $\langle 1\times j,\ j\times 1\rangle$. The preimage of $\widetilde{D}$ by $g\times g$ consists of four divisors:
\begin{equation}\label{preGammadihe}
\begin{array}{l}
D_0=(1\times 1)(D)=\{(x,i(x)),\ x\in D\},\\
D_1=(1\times j)(D)=\{(x,ji(x)),\ x\in D\},\\
D_2=(j\times 1)(D)=\{(j(x),i(x)),\ x\in D\}=\{(x,ij(x)),\ x\in D\}\hspace{5pt}\textrm{ and}\\
D_3=(j\times j)(D)=\{(j(x),ji(x)),\ x\in D\}=\{(x,jij(x)),\ x\in D\}.
\end{array}
\end{equation}

The curve $D_0$ is non-singular and the morphism $g\times g$ restricted to $D_0$ is a local isomorphism. Indeed, it would fail to be a local isomorphism for those $x\in D$ with $x=j(x)$ and $i(x)=ji(x)$, but there are no such points by Lemma \ref{cycdihe}. Therefore, the only singularities of $\widetilde{D}$ are points with more than one preimage on $D_0$.

Two points in $D_0$, $(x,i(x))$ and $(y,i(y))$, have the same image by $g\times g$ if and only if $x=j(y)$, $x=(ij)^2x$ and $y=(ij)^2y$. Therefore, for each two points fixed by $(ij)^2$ there is a singularity in $\widetilde{D}$, and these are all the singularities of $\widetilde{D}$. The points $(x,i(x))$ and $(y,i(y))$ are two intersections of the curves $D_0$ and $D_3$ with the same image by $g\times g$. Since $g\times g$ is not ramified in these points and the divisors $D_0$ and $D_3$ are transversal by Corollary \ref{grtrans}, we deduce that $\widetilde{D}$ is transversal on the image, and therefore its singularities are nodes.

Now, $\widetilde{B}$ is the image of $\widetilde{D}$ by $\pi_C$. Since $\pi_C$ is a double covering ramified at the diagonal, it is a local isomorphism outside the diagonal $\Delta_C$. Therefore, for each two nodes of $\widetilde{D}$ outside the diagonal, there is one node in $\widetilde{B}$. 

A singularity of $\widetilde{D}$ over $\Delta_C$ is the image of two points in $D_0$ corresponding to $x\in D$ such that $x=ij(x)$, that is, $(x,i(x))$ and $(i(x),x)$. Since it lays over the diagonal, there are no other points with the same image $g(x)+g(x)$. Moreover, the points $(x,i(x))$ and $(i(x), x)$ have the same image by $\pi_{D}$, and hence $g(x)+g(x)$ has only one preimage in $B$, which is the normalization of $\widetilde{B}$, and $\widetilde{B}$ has only one branch in $g(x)+g(x)$.
 
Consequently, $(g(x),g(x))$ is a nodal singularity in $\widetilde{D}\subset C\times C$ which image in $\widetilde{B}\subset C^{(2)}$ has a single branch. We want to see that this branch is smooth. 

Let $(z_1,z_2)$ be a system of local coordinates in $C\times C$ with both $z_i$ a local coordinate in $C$ around $g(x)$. Using them, $\pi_C$ is written locally as $(z_1,z_2)\mapsto (z_1+z_2, z_1z_2)=(z,t)$ with $(z,t)$ local coordinates in $C^{(2)}$ centered at $g(x)+g(x)$. Making a local computation and using that in $(g(x),g(x))$ there is a node, we obtain that $g(x)+g(x)$ is a smooth point of $\widetilde{B}$.  

Since the intersection multiplicity in $(g(x),g(x))$ of $\widetilde{D}$ and $\Delta_{C\times C}$ is two, also the intersection multiplicity in $g(x)+g(x)$ of $\widetilde{B}$ and $\Delta_C$ is two, and therefore these two curves are tangent at this point. 
\end{proof}

Therefore, we obtain the following corollary.

\begin{corollary}\label{genBdieh} For $\widetilde{B}$ and $\widetilde{D}$ as above, $p_a(\widetilde{B})-g(B)=\dfrac{1}{4} (\nu((ij)^2)-\nu(ij))$  and $p_a(\widetilde{D})-g(D)=\dfrac{1}{2}\nu((ij)^2)$.
\end{corollary}
  
\section{Proof of Theorem \ref{clasdeg2}}\label{proofthm}

Next, we study degree two curves with $\widetilde{B}^2>0$ in order to prove Theorem \ref{clasdeg2}. We divide the proof in two parts depending on whether $\pi_C^*(\widetilde{B})$ is reducible or irreducible.

\subsection{Reducible case}

We begin the reducible case with an example.

\begin{example}\label{exred} 
We consider the $(0; 10,5,2)$-generating vector of $\mathbb{Z}/10$ given by $(\alpha, \alpha^4, \alpha^5)$ where $\alpha$ denotes a generator of the group. It defines a morphism $C\rightarrow \mathbb{P}^1$, where $C$ is a curve with $\alpha \in {\rm Aut}(C)$ such that $\nu(\alpha)=1$, $\nu(\alpha^2)=3$ and $\nu(\alpha^5)=6$ (see Lemma \ref{orbit}). By the Riemann-Hurwitz formula we obtain that 
\[
2g(C)-2=10(-2)+ 1 \cdot 9+2\cdot 4+5\cdot 1 \Rightarrow g(C)=2.
\]

Consider now the graph of $\alpha$, $\Gamma_{\alpha}$, in $C\times C$. The image of $\Gamma_{\alpha}$ by $\pi_C$ is a curve $\widetilde{B}$ of degree two in $C^{(2)}$, such that $\pi_C^*(\widetilde{B})=\Gamma_{\alpha}+\Gamma_{\alpha^{-1}}$. 
\end{example} 

In the following lemma we prove that it is the only instance of a curve of degree two in a symmetric square $C^{(2)}$ with reducible preimage by $\pi_C$ and positive self-intersection.  

\begin{lemma}\label{redcase}
Let $(C, \widetilde{B})$ be a pair of curves with $C$ smooth and $\widetilde{B}\subset C^{(2)}$ of degree two, such that $\pi_C^*(\widetilde{B})$ is reducible and $\widetilde{B}^2>0$. Then, $C$ is a curve of genus $2$ with an automorphism $\alpha\in {\rm Aut}(C)$ such that $o(\alpha)=10$, $\nu(\alpha)=1$, $\nu(\alpha^2)=3$, $\nu(\alpha^5)=6$, and $\widetilde{B}$ is the symmetrization of the graph of $\alpha$.
\end{lemma}

\begin{proof}

From Proposition \ref{redtwo} we have that $\widetilde{B}=\pi_C(\Gamma_{\alpha})$ for $\alpha\in {\rm Aut}(C)$, $C$ is the normalization of $\widetilde{B}$, and $p_a(\widetilde{B})=g(C)+\frac{1}{2} (\nu(\alpha ^2)-\nu(\alpha))$.  Since necessarily $p_a(\widetilde{B})>g(C)$ we obtain that $\widetilde{B}$ is singular, so $o(\alpha)$ is even (different from $2$) and $\nu(\alpha ^2)-\nu(\alpha)$ is an even number. Moreover, by \cite[V.1.5]{Far} we know that 
\begin{equation} \label{pford} 
\nu(\alpha^2)\leq 2+\frac{2g}{o(\alpha^2)-1}.
\end{equation}

We call $g=g(C)$, $s=\nu(\alpha)$ and $r=\nu(\alpha ^2)-\nu(\alpha)$. From the adjunction formula we deduce that  
\begin{align}
\widetilde{B}^2=2p_a(\widetilde{B})-2-2(2g-2)+\widetilde{B}\cdot \frac{\Delta_C}{2}&=-2g+2+r+s>0 \Leftrightarrow \nonumber \\
 2g-2<r+s&=\nu(\alpha^2).\label{psi}
\end{align}

Next, we consider the different possibilities for $o(\alpha)$ and hence $o(\alpha^2)$:

If $o(\alpha)=6$, then by \eqref{pford} and \eqref{psi} we deduce that $g=3$ and $r+s=5$ with $r$ even. With this conditions, the Riemann-Hurwitz formula \eqref{RHG} is not satisfied for $C\rightarrow C/\langle \alpha \rangle$.

If $o(\alpha)=4$, then by \eqref{psi} we deduce $r+s\geq 2g-1$. By \eqref{pford} and \eqref{RHG}, we deduce $g(C/\langle \alpha \rangle)=0$ and $s=3$, that is not compatible with $r$ even. Hence, such an action does not exist.

If  $o(\alpha)\geq 8$ (and $o(\alpha^2)\geq 4$), then from \eqref{pford} and \eqref{psi} we deduce that
\begin{equation}\label{inord}
2+\frac{2g}{3}\geq 2+\frac{2g}{o(\alpha^2)-1}\geq r+s > 2g-2.
\end{equation}
This implies that $3>g$, so it remains to consider $g=2$ with $r+s=3$. From \eqref{inord} we deduce that $5 \geq o(\alpha^2)\geq 4$. Since $r$ is even and different from zero, we have that $r=2$ and $s=1$. There could be other ramification points coming from points fixed by $\alpha^{o(\alpha^2)}$ that would come by groups of $o(\alpha^2)$ (see Lemma \ref{orbit}). Considering the Riemann-Hurwitz formula $C\rightarrow C/\langle \alpha \rangle$ we obtain that the only compatible case is $o(\alpha^2)=5$ and $C/\langle \alpha \rangle = \mathbb{P}^1$ with $\nu(\alpha^5)=5+1=6$, that is, the case described in Example \ref{exred}.   
\end{proof}

The Hurwitz space of such morphisms $C\rightarrow \mathbb{P}^1=C/\langle \alpha \rangle$ is $0$-dimensional. Indeed, there are only three branch points, that by the action of the automorphisms of $\mathbb{P}^1$, can be fixed as $0,1,\infty$. Therefore, in moduli, we have a finite number of such curves $C$ (see \cite{Serre}). We have proven point $0.$ of Theorem \ref{clasdeg2}.

\subsection{Irreducible case}

Let $(C, \widetilde{B})$ be a pair of curves with $C$ smooth and $\widetilde{B}\subset C^{(2)}$ of degree two with $\widetilde{D}:=\pi_C^*(\widetilde{B})$ irreducible. Let $B$ be the normalization of $\widetilde{B}$ and $D$ the normalization of $\widetilde{D}$. By the arguments in Section \ref{dihe} the curves $B$, $C$ and $D$ lay in a diagram 
\begin{equation}\label{diagdihe}
\xymatrix{
D \ar[r]^{(2:1)} \ar[d]^{(2:1)}& B \\
C &
}
\end{equation}
that does not complete. There exist two involutions $i$ and $j$ (the changes of sheet) such that $B=D/\langle i\rangle $ and $C=D/\langle j\rangle $ with $\langle i,j \rangle =D_n$.

We use the following notation for the number of fixed points of some of the automorphisms of the curve $D$: 
\[
s=\nu(j),\hspace{15pt} t=\nu(i),\hspace{15pt} r=\nu(ij) \hspace{10pt}\textrm{and}\hspace{10pt} r+k=\nu((ij)^2). 
\] 
Let $b=g(B)$, $g=g(C)$ and $h=g(D)$. 

The strategy is to find the restrictions on the numbers $s,t,r,k, b, g$ and $h$ given by our hypothesis.

First, by the Riemann-Hurwitz formula for the morphism $D \rightarrow C$ we obtain
\begin{equation}\label{2a}
g=\frac{2h+2-s}{4}.
\end{equation}

Second, by the Riemann-Hurwitz formula for the morphism $D \rightarrow B$ and Corollary \ref{genBdieh} we deduce that
\begin{equation}\label{2c}
p_a(\widetilde{B})=b+\frac{1}{4}(r+k-r)=\frac{1}{4}(2h+2-t+k).
\end{equation}

Third, by \eqref{2a} and \eqref{2c} the condition $g<p_a(\widetilde{B})$ translates into
\begin{equation}\label{2d}
\frac{2h+2-s}{4}<\frac{2h+2-t}{4}+\frac{1}{4}k \Leftrightarrow t<s+k.
\end{equation}

Finally, from Lemma \ref{selfint} and Corollary \ref{genBdieh} together with \eqref{2a} we deduce that 
\begin{equation}\label{2b}
\widetilde{B}^2=-h+1+s+\frac{1}{2}(r+k)>0  \Leftrightarrow h\leq s+ \frac{1}{2}(r+k).
\end{equation}

\vspace{5pt}
\textbf{Note:} \emph{We can assume that $\langle i,j\rangle =D_{2l}$}. Indeed, if $\langle i,j\rangle =D_{2l+1}$, then the involutions $i$ and $j$ would be conjugate, and so $t=s$. Since $2$ and $2l+1$ are coprime, the automorphisms $ij$ and $(ij)^2$ would have the same fixed points and thus $k=0$, contradicting \eqref{2d}. 
\vspace{5pt}

Let $\gamma=g(D/D_{2l})$. By the Riemann-Hurwitz formula for group quotients (see \eqref{RHG}) applied to $\tau:D \rightarrow D/D_{2l}$, we have that
\begin{equation}\label{RH}
h-1=2l(2\gamma-2)+2l\sum_{P\in \mathrm{Br}}\left(1-\frac{1}{m_P}\right)
\end{equation}
where $m_P=|G_Q|$ with $\tau(Q)=P$.

We count the number of branch points of $\tau$ using Lemma \ref{orbit} repetitively. Since $i$ and $j$ are non conjugate and have order two, there are $\frac{t}{2}$ branch points of order $2$ corresponding to the conjugacy class of the stabilizer $\langle i \rangle$. There are $\frac{s}{2}$ branch points of order $2$ corresponding to the conjugacy class of $\langle j \rangle $, with no point common to the previous set. Moreover, there are $\frac{r}{2}$ branch points of order $2l$ corresponding to $\langle ij \rangle$, and $\frac{k}{4}$ branch points of order $l$ corresponding to $\langle (ij)^2\rangle$. We could also have other branch points coming from powers of $ij$ that do not generate the whole $\langle ij \rangle$. Note that $s$, $t$ and $r$ are even and $k$ is multiple of $4$.

All together, this gives that
\begin{align*}
2l\sum\limits_{P\in Br} \left(1-\frac{1}{m_P}\right)&\geq 2l(\frac{s+t}{2}\cdot \frac{1}{2}+\frac{r}{2}(1-\frac{1}{2l})+\frac{k}{4}(1-\frac{1}{l}))\\
&=\frac{l}{2}(s+t)+r\frac{2l-1}{2}+k\frac{l-1}{2}.
\end{align*}

Then, by \eqref{RH} we have that
\begin{equation} \label{2}
h-1\geq 2l(2\gamma-2)+(s+t)\frac{l}{2}+r\frac{2l-1}{2}+k\frac{l-1}{2}.
\end{equation}
 
By \eqref{2b} and \eqref{2}, and using that $l\geq 2$ we deduce that
\begin{equation} \label{eq:claim}
4(2\gamma-2)+(s+t)+r\frac{3}{2}+k\frac{1}{2} \leq s+ \frac{1}{2}(r+k)-1.
\end{equation}
Hence, $8(\gamma-1)+t+r\leq -1$ and therefore $\gamma=0$ and $t+r\leq 7$. Since $t$ and $r$ are even, we conclude that $t+r\leq 6$. Thus, from \eqref{eq:claim} with $\gamma=0$ we get
\begin{equation}\label{condpg}
l(t+r-8)+(l-2)(s+r+k)< 0.
\end{equation}
For $l\geq 4$ the inequality \eqref{condpg} implies 
\begin{equation}\label{condp}
\frac{s+r+k}{8-(t+r)}<\frac{l}{l-2}\leq 2.
\end{equation}
These inequalities will allow us to reduce the possibilities for $(s,r,t,k)$ to a finite list for $l\geq 3$. We are going to study separately the cases $l\geq 4$, $l=3$ and $l=2$.

We summarize the results from \cite{Ca} on dihedral covers of $\mathbb{P}^1$ that we use in the following theorem.
\begin{theorem}\label{CataneseThm}
Dihedral covers of $\mathbb{P}^1$ of a fixed numerical type $(0; m_1, \dots, m_n)$ form an irreducible closed subvariety of dimension $n-3$ whose image on the moduli space $\mathcal{M}_g$, for $g$ as in \eqref{RHG}, is an irreducible closed subset of the same dimension.
\end{theorem}

We remark that the theorem applies to those numerical types that define a dihedral cover, that is, that satisfy the Riemann Existence Theorem.

We want to give a generating vector of a dihedral group $D_{2l}$ in such a way that it defines an action on a curve $D$ with exactly the number of fixed points determined by $(t,r,s,k)$, and possibly some other coming from other powers of the automorphism $ij$. Since we have imposed that $D/D_{2l} =\mathbb{P}^1$, the type of the vector will be $(0; m_1, \dots, m_n)$. We omit the zero from now on and write only $(m_1, \dots, m_n)$. 

By Riemann's Existence Theorem, having such a vector we can conclude that the family corresponding to that numerical type is non-empty. There are several possible choices of generating vectors with the same numerical type; we give one concrete possibility to prove existence, and then by Theorem \ref{CataneseThm} we get an irreducible family of dihedral covers. Once we have such an action, we can construct a non completing diagram \eqref{diagdihe} satisfying all our conditions with $C=D/\langle j \rangle$ and $B=D/\langle i \rangle$.

We study the properties of the curves in each family. In particular, we compute the genus of $D$, $C$ and $B$, the arithmetic genus of $\widetilde{B}$ and its self-intersection. Moreover, we compute the dimension of each family, with special attention to the dimension of the image in the moduli space of curves of genus $g(C)$. 

To begin with, we prove the following lemma.

\begin{lemma}\label{pfquoc}
Let $D$ be a curve with an action of $D_{2l}=\langle i,j \rangle$ where $i$ and $j$ are two involutions. Let $C=D/\langle j\rangle $ and $\beta$ be the automorphism in $C$ induced by $(ij)^l$. Then, $\nu(\beta)=\frac{1}{2}(\nu(i)+\nu((ij)^l))$ if $l$ is odd, and $\nu(\beta)=\frac{1}{2}(\nu(j)+\nu((ij)^l))$ if $l$ is even.
\end{lemma}  

\begin{proof}

Since $(ij)^l$ is in the center of $D_{2l}$, its action descends to $C$ and $\beta$ is well defined.

Consider $C$ embedded in $D^{(2)}$ as $\{P+j(P),\ P\in D\}$. In this way, the action of $\beta$ on $C$ is
\[
\beta(P+j(P))=(ij)^l(P)+(ij)^l(j(P))=(ij)^l(P)+j((ij)^l(P)).
\] 

A point is fixed by $\beta$ when either $P=(ij)^l(P)$ or $P=(ij)^lj(P)$. From the former we obtain $\frac{1}{2}\nu\left((ij)^l\right)$ points of $C$ fixed by $\beta$ and from the later $\frac{1}{2}\nu\left((ij)^lj\right)$. When $l$ is odd $(ij)^lj$ and $i$ are conjugate in $D_{2l}$, so $\nu((ij)^lj)=\nu(i)$ and when $l$ is even $(ij)^lj$ and $j$ are conjugate in $D_{2l}$ and hence $\nu((ij)^lj)=\nu(j)$.
\end{proof} 

We will denote by $\beta_C\in \mathrm{Aut}(C)$ the action on $C$ induced by $(ij)^{l}$ and by $\beta_B$ the action induced on $B$.

By the discussion in Section \ref{dihe}, knowing the number of fixed points corresponding to the different conjugacy classes in $D_{2l}$, the computation of the genus of $D$, $C$ and $B$, the arithmetic genus of $\widetilde{B}$ and the self-intersection of $\widetilde{B}$ is straight forward. We do not include these computations. In a similar way, the conclusions about $C$ and $B$ being hyperelliptic or bielliptic come from the action of $\beta_C$ or $\beta_B$; the details are omitted since the arguments are based in a repeated use of Riemann-Hurwitz Theorem and Lemma \ref{pfquoc}.

We call $\mathcal{D}$ the irreducible variety parametrizing the dihedral covers of $\mathbb{P}^1$ of a fixed numerical type. The image of $\mathcal{D}$ in the moduli space $\mathcal{M}_h$, given by forgetting the action, is an irreducible variety of the same dimension (Theorem \ref{CataneseThm}). We want to study the morphism $\eta$ from $\mathcal{D}$ to $\mathcal{M}_g$ that sends $(D, \rho)$ to $[C]$, and we wonder in which cases it has positive dimensional fibers. 

We study each numerical case $(t,r,s,k)$ separately to finish the proof of Theorem \ref{clasdeg2}. We give some details in the first case and omit them for the rest of cases. Before, we make some general remarks.

First, we consider the morphism $q:D/\langle j, (ij)^l\rangle \rightarrow D/D_{2l}\cong \mathbb{P}^1$. We observe that for $l\neq 2$ it is not Galois since $\langle j, (ij)^l\rangle$ is not normal in $D_{2l}$, and it is Galois for $l=2$.

Second, to give a curve $D$ with an action of $D_{2l}$ is equivalent to give $\mathbb{P}^1$ with a certain number, $n$, of marked points, and the branching data for the map. To avoid automorphisms, we can fix three of these points to be $0$, $1$ and $\infty$. As we change the rest of points, we change the pair $(D, \rho)$ in the family $\mathcal{D}$. 

Third, we will show that in all cases the curve $C$ is $\gamma$-hyper\-elliptic for $\gamma=0,1$ with $\beta_C=i_{\gamma}$, hence, to give the curve $C$ is equivalent to give $\mathbb{P}^1$ or the curve $E$, with the branch points of $p$ (the $\gamma$-hyperelliptic morphism) marked ($m$ points). 

We have the following diagram of curves for each described action $\rho$ of $D_{2l}=\langle i,j \rangle$ on a curve $D$:
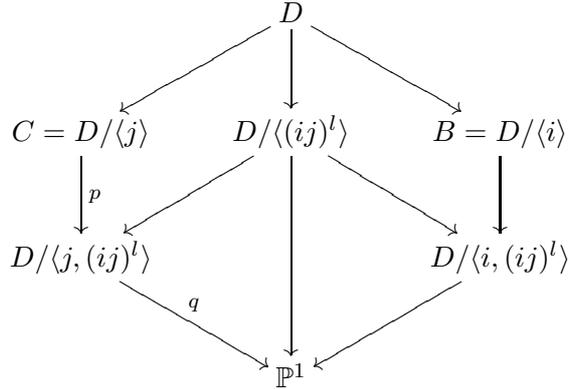
\[
\xymatrix@R+5pt{
& D \ar[dl] \ar[d] \ar[dr] &\\
C=D/\langle j \rangle \ar[d]^{p}& D/\langle (ij)^l\rangle \ar[dl] \ar[dd] \ar[dr] & B=D/\langle i \rangle \ar[d]\\
D/\langle j, (ij)^l \rangle \ar[dr]^{q} && D/\langle i,(ij)^l \rangle \ar[dl]\\
& \mathbb{P}^1 &
}
\]
We observe that for $l\neq 2$ the curves $D/\langle j, (ij)^l$ and $D/\langle i,(ij)^l \rangle$ are isomorphic because $\langle j, (ij)^l$ and $\langle i,(ij)^l \rangle$ are conjugate.

In some cases we add a figure showing the ramification of the morphisms $p$ and $q$ in order to clarify the arguments. The points marked with a diamond are ramification and branch points for $p$. Those marked with a bow-tie are ramification and branch points for $q$. The points marked with a cross are both branch and ramification points for $p$ and $q$ respectively.

Note that the morphism $\eta$ (that maps $(D,\rho)\in \mathcal{D}$ to $[C]\in\mathcal{M}_g$) is equivalent to send the curve determined by the data 
\[
\{\mathbb{P}^1; 0,1,\infty, x_1,\dots, x_{n-3}\}
\]
and the monodromy description, to the curve determined by the data $\{F; x_1,\dots, x_m\}$ were $F$ is the genus $\gamma$ curve given by the quotient of $C$ by its $\gamma$-hyperelliptic involution. Therefore, studying the fibers of $\eta$ is equivalent to studying the fibers of the morphism $\mathcal{M}_{0,n}\times \{\rho\} \rightarrow \mathcal{M}_{\gamma, m}$ defined by the previous correspondence. 

Finally, given a curve $[C]$ in the image of $\eta$, we consider the data determined by its $\gamma$-hyperelliptic involution, which is unique for $\gamma=0$ and there is at most a finite number of possibilities for $\gamma=1$. If we know the morphism $q$, we can recover the data that determines $(D, \rho)$ by taking the images of the branch points of $p$ together with the rest of branch points of $q$. Therefore, we translate the question on the dimension of the fiber of $\eta$ to a question on the number of possible morphisms $q$ for a given curve $[C]$ in the image of $\eta$. Given $\mathbb{P}^1$ (respectively $E$) with $m$ marked points and some information on the branching type of $q$, we want to determine whether there are a finite number of possible $q$'s and a finite number of curves $[C]$, or there is a positive dimensional family of $q$'s and $\eta$ has positive dimensional fibers.

We recall that by the discussion above we have the following numerical conditions:
\begin{align}
h\leq s+ \frac{1}{2}(r+k) \tag{\ref{2b}}\\
l(t+r-8)+(l-2)(s+r+k)< 0 \tag{\ref{condpg}}\\
\frac{s+r+k}{8-(t+r)}<\frac{l}{l-2}\leq 2 (\text{ for } l\geq 4)\tag{\ref{condp}}\\
\gamma=0\\
t+r\leq 6 \label{tr6}\\
r,\ s,\ t \text{ even};\ k \text{ multiple of }4 \nonumber
\end{align}

\vspace{5pt}

\textbf{Assume first that $\boldsymbol{l\geq 4}$.} By conditions \eqref{2b}, \eqref{condp} and \eqref{tr6} together with the parities of $r$, $s$, $t$ and $k$, the possibilities for $(t,r,s,k)$ are $28$:
\[
\begin{array}{lllllllll}
t=0& r=0 & s=0 & k\in\{0,4,8,12\} &&t=2 & r=0 & s=2 & k\in\{0,4,8\}\\
t=0&r=0&s=4 &k\in\{ 0,4,8\}&&t=2&r=0&s=6 & k\in\{0,4\}\\
t=0&r=0&s=8 &k\in\{ 0,4\}&& t=2 &r=0&s=10 & k=0\\
 t=0&r=0&s=12 & k=0&& t=2  &r=2 & s=2 & k=0\\
t=0& r=2&s=0& k\in\{0,4,8\} && t= 4 & r=0 & s=0 & k\in\{0,4\}\\
t=0&r=2&s=4& k\in\{0,4\} && t=4 &r=0&s=4 & k=0\\
 t=0&r=2&s=8 & k=0  && t=4 &r=2& s=0 &k= 0 \\
	t=0&r=4&s= 0 & k=0&&&&&
\end{array}
\]

By the Riemann Hurwitz formula for $D\rightarrow D/D_{2l}=\mathbb{P}^1$ and considering that we need a compatible generating vector, we reduce this case to three possibilities (see \cite{Tesi} for a detailed discussion).

(D10.1)  $\boldsymbol{t=0}$, $\boldsymbol{r=2}$, $\boldsymbol{s=4}$, $\boldsymbol{k=0}$.

By the Riemann-Hurwitz formula for $D\rightarrow D/D_{2l}=\mathbb{P}^1$, we find that there is another branch point with $m_i=2$ and also that $l=5$. Thus, the additional branch point corresponds to points fixed by $(ij)^5\in D_{10}$.  

Consequently, we impose one branch point coming from $ij$ (image of points fixed by it), $2$ branch points coming from $j$ (and its conjugates $(ij)^{2\alpha}j$) and one coming from $(ij)^5$. 

The $(10,10,2,2)$-generating vector of $D_{10}$ given by $\left(ij, (ij)^5, (ij)^4j, j\right)$ satisfies all conditions. 

We observe that $\mathcal{D}$ is the $1$-dimensional family of all curves of genus $5$ with maximal dihedral symmetry (see \cite{Buj}). 

\textbf{Claim:} The map $\eta$ is finite.

\emph{Proof of the claim:} First, since the morphism $q:\mathbb{P}^1 \rightarrow \mathbb{P}^1$ has $3$ branch points, there are a finite number of such morphisms modulo automorphisms of $\mathbb{P}^1$. 

Second, we observe that $D/D_{10}\cong \mathbb{P}^1$ has four marked points. To avoid automorphisms we fix $x$ to be the branch point associated with (image of) points fixed by $(ij)^5$, $0$ to be associated with the points fixed by $ij$ and $\{1, \infty\}$ to be associated with $j$.

Third, $C/\langle \beta_C \rangle \cong \mathbb{P}^1$ has six marked points, the points where the hyperelliptic morphism $p$ is branched. We observe that since $p$ is the projection given by the action of $\beta_C$, five of the branch points of $p$ are a fiber of $q$, in particular, the images of points fixed by $(ij)^5$ in $D$, and the sixth has ramification index $5$ in $q$, in particular, the image of the points fixed by $ij$.

In the following figure we show the ramification of the morphisms $p$ and $q$. 
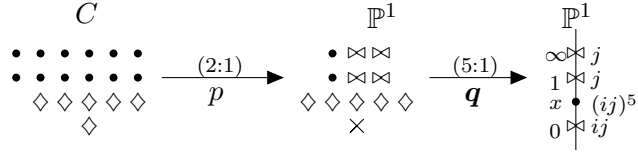
\begin{center}
\begin{tikzpicture}[line cap=round,line join=round,>=triangle 45,x=0.32cm,y=0.32cm]
\clip(-10,-5) rectangle (19,1.5);
\draw (16,0)-- (16,-5);
\draw (15,1.5) node[anchor=north west] {$\mathbb{P}^1$};
\draw (7,1.5) node[anchor=north west] {$\mathbb{P}^1$};
\draw (-5,1.5) node[anchor=north west] {$C$};
\draw [->,color=black] (-1,-2) -- (4,-2);
\draw [->] (10,-2) -- (14,-2);
\draw (0,-0.7) node[anchor=north west] {$_{(2:1)}$};
\draw (10.5,-0.7) node[anchor=north west] {$_{(5:1)}$};
\draw (11,-2) node[anchor=north west] {$\boldsymbol{q}$};
\draw (0.5,-2) node[anchor=north west] {$p$};
\draw (16.2,-3.3) node[anchor=north west] {$_{ij}$};
\draw (16.2,-2.3) node[anchor=north west] {$_{(ij)^5}$};
\draw (16.2,-1.3) node[anchor=north west] {$_j$};
\draw (16.2,-0.3) node[anchor=north west] {$_j$};
\draw (14.5,-3.5) node[anchor=north west] {$_0$};
\draw (14.5,-2.5) node[anchor=north west] {$_x$};
\draw (14.5,-1.5) node[anchor=north west] {$_1$};
\draw (14.3,-0.5) node[anchor=north west] {$_{\infty}$};
\begin{scriptsize}
\fill [color=black] (-2,-1) circle (1.5pt);
\fill [color=black] (-3,-1) circle (1.5pt);
\fill [color=black] (-4,-1) circle (1.5pt);
\fill [color=black] (-5,-1) circle (1.5pt);
\fill [color=black] (-6,-1) circle (1.5pt);
\fill [color=black] (-7,-1) circle (1.5pt);
\fill [color=black] (-7,-2) circle (1.5pt);
\fill [color=black] (-6,-2) circle (1.5pt);
\fill [color=black] (-5,-2) circle (1.5pt);
\fill [color=black] (-4,-2) circle (1.5pt);
\fill [color=black] (-3,-2) circle (1.5pt);
\fill [color=black] (-2,-2) circle (1.5pt);
\draw (-2,-3) node {$\diamondsuit$};
\draw  (-3,-3) node {$\diamondsuit$};
\draw  (-4,-3) node {$\diamondsuit$};
\draw  (-5,-3) node {$\diamondsuit$};
\draw  (-6,-3) node {$\diamondsuit$};
\draw  (-4,-4) node {$\diamondsuit$};
\draw  (5,-3) node {$\diamondsuit$};
\draw  (6,-3) node {$\diamondsuit$};
\draw  (7,-3) node {$\diamondsuit$};
\draw  (8,-3) node {$\diamondsuit$};
\draw  (9,-3) node {$\diamondsuit$};
\draw  (7,-4)-- ++(-2.5pt,-2.5pt) -- ++(5.0pt,5.0pt) ++(-5.0pt,0) -- ++(5.0pt,-5.0pt);
\fill [color=black] (6,-2) circle (1.5pt);
\draw (7,-2)  node {$\bowtie$};
\draw (8,-2) node {$\bowtie$};
\draw (8,-1) node {$\bowtie$};
\draw (7,-1) node {$\bowtie$};
\fill [color=black] (6,-1) circle (1.5pt);
\draw (16,-4) node {$\bowtie$};
\fill [color=black] (16,-3) circle (1.5pt);
\draw (16,-2) node {$\bowtie$};
\draw (16,-1) node {$\bowtie$};
\end{scriptsize}
\end{tikzpicture}
\end{center}

Therefore, $\eta$ is equivalent to send the curve determined by the data $\{\mathbb{P}^1; 0, 1, \infty, x\}$ to the one determined by $\{\mathbb{P}^1; q^{-1}(0), q^{-1}(x)\}$. 

Finally, given a curve $[C]$ in the image of $\mathcal{D}$, we recover $\eta^{-1}([C])$ taking the image of the branch points of its hyperelliptic involution ($0$ and $x$) by a suitable $q$ together with the other two branch points of $q$ ($1$ and $\infty$). By a suitable $q$ we mean that one of the branch points of $p$ is a ramification point of $q$, and the other five are a fiber of $q$.

Such a morphism $q$ exists because $C$ is in the image of $\mathcal{D}$, and therefore it is the quotient of a $D$. Moreover, there are only a finite number of possibilities for $q$, and hence, we can recover at most a finite number of $(D, \rho)\in \mathcal{D}$. $\Diamond$

(D10.2) $\boldsymbol{t=2}$, $\boldsymbol{r=0}$, $\boldsymbol{s=2}$, $\boldsymbol{k=4}$. 
 
By the Riemann-Hurwitz formula for $D\rightarrow D/D_{2l}=\mathbb{P}^1$, we find that there is one more branch point with $m_i=2$ and also that $l=5$. The $(5,2,2,2)$-generating vector of $D_{10}$ given by $\left((ij)^2, (ij)^5, (ij)^2i, j\right)$ satisfies all conditions.   

We observe that $\mathcal{D}$ is the $1$-dimensional family of all curves of genus $4$ with maximal dihedral symmetry (see \cite{Buj}). 

Moreover, since $q$ has three branch points we deduce, with the arguments used in point (D10.1), that the map $\eta$ is finite.

(D10.3) $\boldsymbol{t=2}$, $\boldsymbol{r=0}$, $\boldsymbol{s=6}$, $\boldsymbol{k=0}$. 

By the Riemann-Hurwitz formula for $D\rightarrow D/D_{2l}=\mathbb{P}^1$, we find that there is another branch point with $m_i=2$ and also that $l=5$. The $(2,2,2,2,2)$-generating vector of $D_{10}$ given by $\left((ij)^5, (ij)^4i, j, j, j\right)$ satisfies all conditions.

\textbf{Conjecture}: \emph{We expect the map $\eta$ to be finite.} Indeed, there should be only a finite number of possibilities for $q$. Since $q$ is a degree five morphism from $\mathbb{P}^1$ to $\mathbb{P}^1$, in homogeneous coordinates it is given by two degree five polynomials. Given five of the branch points of $p$, we assume that their image is $0\in \mathbb{P}^1$ and we have one of the polynomials determined. Assuming that the sixth point has image $\infty\in \mathbb{P}^1$, we obtain one factor of the other polynomial. When we impose the branching type $(1,2,2)$ for each of the four branch points, we obtain a system of twelve equations with five unknowns. 

The resolution of this system of equations has a very high computational cost\footnote{System specifications: Processor: Intel Xeon W3520 @2.67GHz. 4 GB RAM. Using Windows 64 bits and Wolfram Mathematica 9} because of the high degree of the equations involved. We were not able to finish it. Probably a more refined algorithm would be needed. Nevertheless, the high number of equations compared to the number of unknowns takes us to conjecture that this system of equations has a finite number of solutions. If so, we could recover at most a finite number of $(D, \rho)\in \mathcal{D}$. $\Diamond$

Finally notice that in all three cases $p_a(\widetilde{B})=2g(C)-1$.

\vspace{5pt}

\textbf{Assume now that $\boldsymbol{l=3}$.} We assume that $\langle i,j \rangle=D_6$. Since $s,r,k$ and $t$ are even integers, condition \eqref{condpg} is equivalent to
\begin{equation}\label{condaut3}
s+4r+k+3t\leq 22.
\end{equation} 

In $D_6$ there are six conjugacy classes: $[Id],\ [i],\ [j],\ [ij],\ [(ij)^2]\textrm{ and } [(ij)^3]$. We denote by $p=\nu((ij)^3)$. Thus, the Riemann-Hurwitz formula for $D\rightarrow D/D_6$ reads $h=\frac{1}{2}\left(-22+3s+3t+5r+2k+p\right)$. By \eqref{2b} we deduce that 
\begin{equation}\label{condauts}
s+4r+k+3t+p\leq 22.
\end{equation}
Notice that if it is satisfied, then also \eqref{condaut3} is satisfied.
 
Now, we observe that we can embed $D_6$ in $S_6$ in such a way that $i$ is odd and $j$ is even (thus $ij$ is odd, $(ij)^2$ is even and $(ij)^3$ is odd). Since we will need the product one condition for the generating vector, we need to impose $\frac{t}{2}+\frac{r}{2}+\frac{p}{6}$ to be even, or which is the same, $t+r+\frac{p}{3}$ multiple of four. Furthermore, we can also embed $D_6$ in $S_6$ in such a way that $i$ is even and $j$ is odd and hence we need to impose $s+r+\frac{p}{3}$ multiple of four. By this, inequality \eqref{condauts}, our previous conditions and considering that we need a compatible generating vector, we find the following ten possibilities:

(D6.1) $\boldsymbol{t=0}$, $\boldsymbol{r=0}$, $\boldsymbol{p=12}$, $\boldsymbol{s=4}$, $\boldsymbol{k=4}$. 

The generating vector $\left((ij)^3, (ij)^3, (ij)^2, (ij)^4j, j\right)$ satisfies all conditions. 

Since $q$ has three branch points, by the arguments in case (D10.1), we deduce that the map $\eta$ is finite.

(D6.2) $\boldsymbol{t=0}$, $\boldsymbol{r=2}$, $\boldsymbol{p=6}$, $\boldsymbol{s=4}$, $\boldsymbol{k=4}$. 

The generating vector $\left(ij, (ij)^3, (ij)^2, j, j\right)$ satisfies all conditions.

\textbf{Claim:} The map $\eta$ is finite.

\emph{Proof of the claim:} Given a curve $[C]$ in the image of $\mathcal{D}$, we recover $\eta^{-1}([C])$ taking the image of the branch points of its bielliptic involution by a suitable $q$, together with the other three branch points of $q$. By a suitable $q$ we mean that one of the branch points of $p$ has ramification index $3$ in $q$, and the other $3$ are a fiber of $q$.

In the following figure we show the ramification of the morphisms $p$ and $q$. 
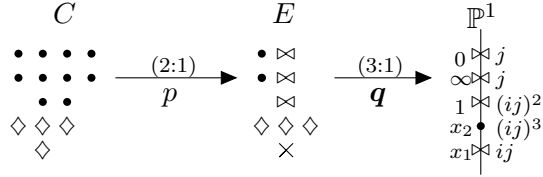
\begin{center}
\begin{tikzpicture}[line cap=round,line join=round,>=triangle 45,x=0.32cm,y=0.32cm]
\clip(-4,-6) rectangle (19,1.5);
\draw (16,0)-- (16,-6);
\draw (15,1.5) node[anchor=north west] {$\mathbb{P}^1$};
\draw (7,1.5) node[anchor=north west] {$E$};
\draw (-2,1.5) node[anchor=north west] {$C$};
\draw [->,color=black] (1,-2) -- (6,-2);
\draw [->] (10,-2) -- (14,-2);
\draw (2,-0.7) node[anchor=north west] {$_{(2:1)}$};
\draw (10.5,-0.7) node[anchor=north west] {$_{(3:1)}$};
\draw (11,-2) node[anchor=north west] {$\boldsymbol{q}$};
\draw (2.5,-2) node[anchor=north west] {$p$};
\draw (16.2,-3.3) node[anchor=north west] {$_{(ij)^3}$};
\draw (16.2,-2.3) node[anchor=north west] {$_{(ij)^2}$};
\draw (16.2,-1.3) node[anchor=north west] {$_j$};
\draw (16.2,-0.3) node[anchor=north west] {$_j$};
\draw (16.2,-4.3) node[anchor=north west] {$_{ij}$};
\draw (14.3,-3.5) node[anchor=north west] {$_{x_2}$};
\draw (14.5,-2.5) node[anchor=north west] {$_1$};
\draw (14.3,-1.5) node[anchor=north west] {$_{\infty}$};
\draw (14.5,-0.5) node[anchor=north west] {$_{0}$};
\draw (14.3,-4.5) node[anchor=north west] {$_{x_1}$};
\begin{scriptsize}
\fill [color=black] (0,-1) circle (1.5pt);
\fill [color=black] (-1,-1) circle (1.5pt);
\fill [color=black] (-2,-1) circle (1.5pt);
\fill [color=black] (-3,-1) circle (1.5pt);
\fill [color=black] (-3,-2) circle (1.5pt);
\fill [color=black] (-2,-2) circle (1.5pt);
\fill [color=black] (-1,-2) circle (1.5pt);
\fill [color=black] (0,-2) circle (1.5pt);
\fill [color=black] (-2,-3) circle (1.5pt);
\fill [color=black] (-1,-3) circle (1.5pt);
\draw  (-1,-4) node {$\diamondsuit$};
\draw  (-2,-4) node {$\diamondsuit$};
\draw  (-3,-4) node {$\diamondsuit$};
\draw  (-2,-5) node {$\diamondsuit$};
\draw  (7,-4) node {$\diamondsuit$};
\draw  (9,-4) node {$\diamondsuit$};
\draw  (8,-4) node {$\diamondsuit$};
\draw  (8,-5)-- ++(-2.5pt,-2.5pt) -- ++(5.0pt,5.0pt) ++(-5.0pt,0) -- ++(5.0pt,-5.0pt);
\fill [color=black] (7,-2) circle (1.5pt);
\draw (8,-2) node {$\bowtie$};
\draw (8,-1) node {$\bowtie$};
\draw (8,-3) node {$\bowtie$};
\fill [color=black] (7,-1) circle (1.5pt);
\fill [color=black] (7,-2) circle (1.5pt);
\draw (16,-3) node {$\bowtie$};
\fill [color=black] (16,-4) circle (1.5pt);
\draw (16,-2) node {$\bowtie$};
\draw (16,-1) node {$\bowtie$};
\draw (16,-5) node {$\bowtie$};
\end{scriptsize}
\end{tikzpicture}
\end{center}

Such a morphism $q$ exists by construction, and there are only a finite number of possibilities for $q$. Indeed, since the elliptic curve is given, one of the branch points of $p$ determines the immersion of $E$ in $\mathbb{P}^2$ in such a way that it is an inflexion point. For at least one of the four possible immersions the other three points lay over a line. Then, the projection point is the intersection of the tangent to the inflexion and the line containing the other three. $\Diamond$

(D6.3) $\boldsymbol{t=0}$, $\boldsymbol{r=2}$, $\boldsymbol{p=6}$, $\boldsymbol{s=8}$, $\boldsymbol{k=0}$. 

The generating vector $\left(ij, (ij)^3, (ij)^2j, j,j, j\right)$ satisfies all conditions. 

\textbf{Claim:} The map $\eta$ is finite.

\emph{Proof of the claim:} Given a curve $C$ in the image of $\mathcal{D}$, we recover $\eta^{-1}([C])$ taking the image of the branch points of its bielliptic involution by a suitable $q$ together with the other four branch points of $q$. By a suitable $q$ we mean that one of the branch points of $p$ has ramification order $3$ in $q$, and the other three are a fiber of $q$.

As in case (D6.2) such a morphism $q$ exists and there are only a finite number of possibilities for it. $\Diamond$

(D6.4) $\boldsymbol{t=0}$, $\boldsymbol{r=4}$, $\boldsymbol{p=0}$, $\boldsymbol{s=4}$, $\boldsymbol{k=0}$. 

The generating vector $\left( ij, ji, j, j\right)$ satisfies all conditions.

\textbf{Claim:} The map $\eta$ is finite.

\emph{Proof of the claim:} Given a curve $[C]$ in the image of $\mathcal{D}$, we recover $\eta^{-1}([C])$ taking the image of the branch points of its bielliptic involution by a suitable $q$, together with the other two branch points of $q$. By a suitable $q$ we mean that the two branch points of $p$ have ramification index $3$ in $q$.

In the following figure we show the ramification of the morphisms $p$ and $q$. 
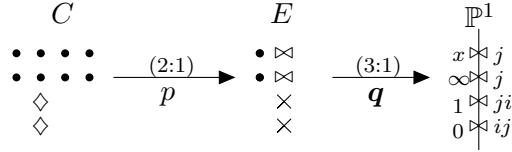
\begin{center}
\begin{tikzpicture}[line cap=round,line join=round,>=triangle 45,x=0.32cm,y=0.32cm]
\clip(-4,-5) rectangle (19,1.5);
\draw (16,0)-- (16,-5);
\draw (15,1.5) node[anchor=north west] {$\mathbb{P}^1$};
\draw (7,1.5) node[anchor=north west] {$E$};
\draw (-2,1.5) node[anchor=north west] {$C$};
\draw [->,color=black] (1,-2) -- (6,-2);
\draw [->] (10,-2) -- (14,-2);
\draw (2,-0.7) node[anchor=north west] {$_{(2:1)}$};
\draw (10.5,-0.7) node[anchor=north west] {$_{(3:1)}$};
\draw (11,-2) node[anchor=north west] {$\boldsymbol{q}$};
\draw (2.5,-2) node[anchor=north west] {$p$};
\draw (16.2,-3.3) node[anchor=north west] {$_{ij}$};
\draw (16.2,-2.3) node[anchor=north west] {$_{ji}$};
\draw (16.2,-1.3) node[anchor=north west] {$_j$};
\draw (16.2,-0.3) node[anchor=north west] {$_j$};
\draw (14.5,-3.5) node[anchor=north west] {$_{0}$};
\draw (14.5,-2.5) node[anchor=north west] {$_1$};
\draw (14.3,-1.5) node[anchor=north west] {$_{\infty}$};
\draw (14.5,-0.5) node[anchor=north west] {$_{x}$};
\begin{scriptsize}
\fill [color=black] (0,-1) circle (1.5pt);
\fill [color=black] (-1,-1) circle (1.5pt);
\fill [color=black] (-2,-1) circle (1.5pt);
\fill [color=black] (-3,-1) circle (1.5pt);
\fill [color=black] (-2,-2) circle (1.5pt);
\fill [color=black] (-1,-2) circle (1.5pt);
\fill [color=black] (0,-2) circle (1.5pt);
\fill [color=black] (-3,-2) circle (1.5pt);
\draw  (-2,-3) node {$\diamondsuit$};
\draw  (-2,-4) node {$\diamondsuit$};
\draw  (8,-3)-- ++(-2.5pt,-2.5pt) -- ++(5.0pt,5.0pt) ++(-5.0pt,0) -- ++(5.0pt,-5.0pt);
\draw  (8,-4)-- ++(-2.5pt,-2.5pt) -- ++(5.0pt,5.0pt) ++(-5.0pt,0) -- ++(5.0pt,-5.0pt);
\fill [color=black] (7,-2) circle (1.5pt);
\fill [color=black] (7,-1) circle (1.5pt);
\draw (8,-2) node {$\bowtie$};
\draw (8,-1) node {$\bowtie$};
\fill [color=black] (7,-1) circle (1.5pt);
\fill [color=black] (7,-2) circle (1.5pt);
\draw (16,-3) node {$\bowtie$};
\draw (16,-4) node {$\bowtie$};
\draw (16,-2) node {$\bowtie$};
\draw (16,-1) node {$\bowtie$};
\end{scriptsize}
\end{tikzpicture}
\end{center}

Such a morphism $q$ exists because $C$ is in the image of $\mathcal{D}$, and therefore it is the quotient of a $D$. Moreover, there are only a finite number of possibilities for $q$. Indeed, since we have the elliptic curve given as the quotient of $C$ by its bielliptic involution, one of the branch points of $p$ determine the immersion of $E$ in $\mathbb{P}^2$ in such a way that it is an inflexion point, and then necessarily the other point will be another inflexion. The projection point will be then the intersection of the respective tangent lines. Thus, we can recover at most a finite number of $(D, \rho)\in \mathcal{D}$. $\Diamond$

(D6.5) $\boldsymbol{t=2}$, $\boldsymbol{r=0}$, $\boldsymbol{p=6}$, $\boldsymbol{s=2}$, $\boldsymbol{k=8}$. 

The generating vector $\left(j, i, (ij)^3, (ij)^2, (ij)^2\right)$ satisfies all conditions. 

\textbf{Claim:} The map $\eta$ is finite.

\emph{Proof of the claim:}  Given a curve $[C]$ in the image of $\mathcal{D}$, we recover $\eta^{-1}([C])$ taking the image of the branch points of its bielliptic involution by a suitable $q$ together with the other three branch points of $q$. By a suitable $q$ we mean that one of the branch points of $p$ is not ramified but lies over a branch point and the other $3$ are a fiber of $q$.

In the following figure we show the ramification of the morphisms $p$ and $q$. 
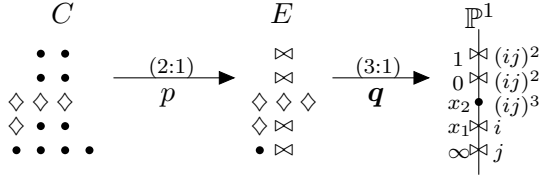
\begin{center}
\begin{tikzpicture}[line cap=round,line join=round,>=triangle 45,x=0.32cm,y=0.32cm]
\clip(-4,-6) rectangle (19,1.5);
\draw (16,0)-- (16,-6);
\draw (15,1.5) node[anchor=north west] {$\mathbb{P}^1$};
\draw (7,1.5) node[anchor=north west] {$E$};
\draw (-2,1.5) node[anchor=north west] {$C$};
\draw [->,color=black] (1,-2) -- (6,-2);
\draw [->] (10,-2) -- (14,-2);
\draw (2,-0.7) node[anchor=north west] {$_{(2:1)}$};
\draw (10.5,-0.7) node[anchor=north west] {$_{(3:1)}$};
\draw (11,-2) node[anchor=north west] {$\boldsymbol{q}$};
\draw (2.5,-2) node[anchor=north west] {$p$};
\draw (16.2,-4.3) node[anchor=north west] {$_j$};
\draw (16.2,-3.3) node[anchor=north west] {$_{i}$};
\draw (16.2,-2.3) node[anchor=north west] {$_{(ij)^3}$};
\draw (16.2,-1.3) node[anchor=north west] {$_{(ij)^2}$};
\draw (16.2,-0.3) node[anchor=north west] {$_{(ij)^2}$};
\draw (14.3,-4.5) node[anchor=north west] {$_{\infty}$};
\draw (14.3,-3.5) node[anchor=north west] {$_{x_1}$};
\draw (14.3,-2.5) node[anchor=north west] {$_{x_2}$};
\draw (14.5,-1.5) node[anchor=north west] {$_0$};
\draw (14.5,-0.5) node[anchor=north west] {$_{1}$};
\begin{scriptsize}
\fill [color=black] (-1,-1) circle (1.5pt);
\fill [color=black] (-2,-1) circle (1.5pt);
\fill [color=black] (-1,-2) circle (1.5pt);
\fill [color=black] (-2,-2) circle (1.5pt);
\fill [color=black] (-2,-4) circle (1.5pt);
\fill [color=black] (-1,-4) circle (1.5pt);
\fill [color=black] (-1,-5) circle (1.5pt);
\fill [color=black] (0,-5) circle (1.5pt);
\fill [color=black] (-2,-5) circle (1.5pt);
\fill [color=black] (-3,-5) circle (1.5pt);
\draw  (-3,-4) node {$\diamondsuit$};
\draw  (-3,-3) node {$\diamondsuit$};
\draw  (-2,-3) node {$\diamondsuit$};
\draw  (-1,-3) node {$\diamondsuit$};
\fill [color=black] (7,-5) circle (1.5pt);
\draw  (7,-4) node {$\diamondsuit$};
\draw  (7,-3) node {$\diamondsuit$};
\draw  (8,-3) node {$\diamondsuit$};
\draw  (9,-3) node {$\diamondsuit$};
\draw (8,-5) node {$\bowtie$};
\draw (8,-4) node {$\bowtie$};
\draw (8,-1) node {$\bowtie$};
\draw (8,-2) node {$\bowtie$};
\draw (16,-5) node {$\bowtie$};
\draw (16,-4) node {$\bowtie$};
\fill [color=black] (16,-3) circle (1.5pt);
\draw (16,-2) node {$\bowtie$};
\draw (16,-1) node {$\bowtie$};
\end{scriptsize}
\end{tikzpicture}
\end{center}

Such a morphism $q$ exists by construction, and there are only a finite number of possibilities for $q$. Indeed, since the elliptic curve is given, taking three of the branch points we determine the immersion of $E$ in $\mathbb{P}^2$, and taking a line passing through the fourth and tangent to $E$ but not on this point (a finite number of such), we obtain a finite number of candidates for the projection point. Only those with two points with ramification index $3$ are possible $q$'s. $\Diamond$

(D6.6) $\boldsymbol{t=2}$, $\boldsymbol{r=0}$, $\boldsymbol{p=6}$, $\boldsymbol{s=6}$, $\boldsymbol{k=4}$.  

The generating vector $\left((ij)^3, (ij)^2, i, j,j, j\right)$ satisfies all conditions.

\textbf{Claim:} The map $\eta$ is finite.

\emph{Proof of the claim:} Given a curve $[C]$ in the image of $\mathcal{D}$, we recover $\eta^{-1}([C])$ taking the image of the branch points of its bielliptic morphism by a suitable $q$ together with the other four branch points of $q$. By a suitable $q$ we mean that one of the branch points of $p$  is non ramified lying over a branch point and the other three are a fiber of $q$.

As in (D6.5) such a morphism $q$ exists and there is a finite number of possibilities for it. $\Diamond$

(D6.7) $\boldsymbol{t=2}$, $\boldsymbol{r=2}$, $\boldsymbol{p=0}$, $\boldsymbol{s=2}$, $\boldsymbol{k=4}$. 

The generating vector $\left(i, j, ij, (ij)^4\right)$ satisfies all conditions.

\textbf{Claim:} The map $\eta$ is finite.

\emph{Proof of the claim:} Given a curve $[C]$ in the image of $\mathcal{D}$, we recover $\eta^{-1}([C])$ taking the image of the branch points of its bielliptic involution by a suitable $q$ and the other two branch points of $q$. By a suitable $q$ we mean that one of the branch points of $p$ has ramification index $3$ in $q$ and the other is non ramified with image a branch point of $q$.

In the following figure we show the ramification of the morphisms $p$ and $q$. 
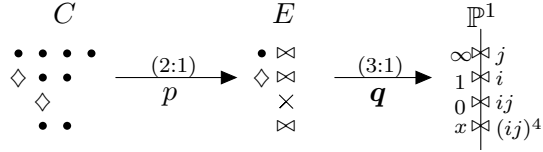
\begin{center}
\begin{tikzpicture}[line cap=round,line join=round,>=triangle 45,x=0.32cm,y=0.32cm]
\clip(-4,-5) rectangle (19,1.5);
\draw (16,0)-- (16,-5);
\draw (15,1.5) node[anchor=north west] {$\mathbb{P}^1$};
\draw (7,1.5) node[anchor=north west] {$E$};
\draw (-2,1.5) node[anchor=north west] {$C$};
\draw [->,color=black] (1,-2) -- (6,-2);
\draw [->] (10,-2) -- (14,-2);
\draw (2,-0.7) node[anchor=north west] {$_{(2:1)}$};
\draw (10.5,-0.7) node[anchor=north west] {$_{(3:1)}$};
\draw (11,-2) node[anchor=north west] {$\boldsymbol{q}$};
\draw (2.5,-2) node[anchor=north west] {$p$};
\draw (16.2,-3.3) node[anchor=north west] {$_{(ij)^4}$};
\draw (16.2,-2.3) node[anchor=north west] {$_{ij}$};
\draw (16.2,-1.3) node[anchor=north west] {$_{i}$};
\draw (16.2,-0.3) node[anchor=north west] {$_{j}$};
\draw (14.5,-3.5) node[anchor=north west] {$_{x}$};
\draw (14.5,-2.5) node[anchor=north west] {$_{0}$};
\draw (14.5,-1.5) node[anchor=north west] {$_1$};
\draw (14.3,-0.5) node[anchor=north west] {$_{\infty}$};
\begin{scriptsize}
\fill [color=black] (-3,-1) circle (1.5pt);
\fill [color=black] (-2,-1) circle (1.5pt);
\fill [color=black] (-1,-1) circle (1.5pt);
\fill [color=black] (0,-1) circle (1.5pt);
\fill [color=black] (-2,-2) circle (1.5pt);
\fill [color=black] (-1,-2) circle (1.5pt);
\fill [color=black] (-1,-4) circle (1.5pt);
\fill [color=black] (-2,-4) circle (1.5pt);
\draw  (-3,-2) node {$\diamondsuit$};
\draw  (-2,-3) node {$\diamondsuit$};
\draw  (8,-3)-- ++(-2.5pt,-2.5pt) -- ++(5.0pt,5.0pt) ++(-5.0pt,0) -- ++(5.0pt,-5.0pt);
\fill [color=black] (7,-1) circle (1.5pt);
\draw  (7,-2) node {$\diamondsuit$};
\draw (8,-1) node {$\bowtie$};
\draw (8,-2) node {$\bowtie$};
\draw (8,-4) node {$\bowtie$};
\draw (16,-3) node {$\bowtie$};
\draw (16,-4) node {$\bowtie$};
\draw (16,-2) node {$\bowtie$};
\draw (16,-1) node {$\bowtie$};
\end{scriptsize}
\end{tikzpicture}
\end{center}

Such a morphism $q$ exists by construction, and there are only a finite number of possibilities for $q$. Indeed, since the elliptic curve is given, one of the branch points determines the immersion on $\mathbb{P}^2$ in such a way that it is an inflexion, and taking  a line passing through the other and tangent to $E$, but not on this point (a finite number of such), we obtain a finite number of candidates for the projection point. Only those with two points with ramification index $3$ would be possible $q$'s. $\Diamond$

(D6.8) $\boldsymbol{t=2}$, $\boldsymbol{r=2}$, $\boldsymbol{p=0}$, $\boldsymbol{s=6}$, $\boldsymbol{k=0}$. 

The generating vector $\left(j,j,j, i, ij\right)$ satisfies all conditions.

\textbf{Claim:} The map $\eta$ is finite.

\emph{Proof of the claim:} Given a curve $[C]$ in the image of $\mathcal{D}$, we recover $\eta^{-1}([C])$ taking the image of the branch points of its bielliptic involution by a suitable $q$ together with the other three branch points of $q$. By a suitable $q$ we mean that one of the branch points of $p$ has ramification index $3$, and the other is non ramified lying over a branch point of $q$.

As in (D6.7) such a morphism $q$ exists and there are only a finite number of possibilities for it. $\Diamond$

(D6.9) $\boldsymbol{t=4}$, $\boldsymbol{r=0}$, $\boldsymbol{p=0}$, $\boldsymbol{s=4}$, $\boldsymbol{k=4}$. 

The generating vector $\left(i, i, (ij)^4, (ij)^2j, j\right)$ satisfies all conditions.

\textbf{Claim:} The map $\eta$ is finite.

\emph{Proof of the claim:} Given a curve $[C]$ in the image of $\mathcal{D}$, we recover $\eta^{-1}([C])$ taking the image of the branch points of its bielliptic involution by a suitable $q$ together with the other three branch points of $q$. By a suitable $q$ we mean that the branch points of $p$ are non ramified and their images by $q$ are different branch points of $q$. 

In the following figure we show the ramification of the morphisms $p$ and $q$. 
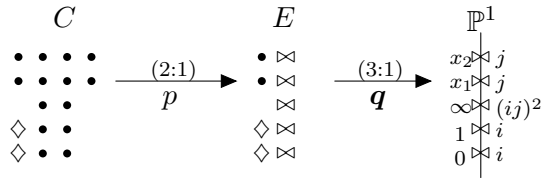
\begin{center}
\begin{tikzpicture}[line cap=round,line join=round,>=triangle 45,x=0.32cm,y=0.32cm]
\clip(-4,-6) rectangle (19,1.5);
\draw (16,0)-- (16,-6);
\draw (15,1.5) node[anchor=north west] {$\mathbb{P}^1$};
\draw (7,1.5) node[anchor=north west] {$E$};
\draw (-2,1.5) node[anchor=north west] {$C$};
\draw [->,color=black] (1,-2) -- (6,-2);
\draw [->] (10,-2) -- (14,-2);
\draw (2,-0.7) node[anchor=north west] {$_{(2:1)}$};
\draw (10.5,-0.7) node[anchor=north west] {$_{(3:1)}$};
\draw (11,-2) node[anchor=north west] {$\boldsymbol{q}$};
\draw (2.5,-2) node[anchor=north west] {$p$};
\draw (16.2,-4.3) node[anchor=north west] {$_{i}$};
\draw (16.2,-3.3) node[anchor=north west] {$_{i}$};
\draw (16.2,-2.3) node[anchor=north west] {$_{(ij)^2}$};
\draw (16.2,-1.3) node[anchor=north west] {$_{j}$};
\draw (16.2,-0.3) node[anchor=north west] {$_{j}$};
\draw (14.5,-4.5) node[anchor=north west] {$_{0}$};
\draw (14.5,-3.5) node[anchor=north west] {$_{1}$};
\draw (14.3,-2.5) node[anchor=north west] {$_{\infty}$};
\draw (14.3,-1.5) node[anchor=north west] {$_{x_1}$};
\draw (14.3,-0.5) node[anchor=north west] {$_{x_2}$};
\begin{scriptsize}
\fill [color=black] (-3,-1) circle (1.5pt);
\fill [color=black] (-3,-2) circle (1.5pt);
\fill [color=black] (-2,-1) circle (1.5pt);
\fill [color=black] (-2,-2) circle (1.5pt);
\fill [color=black] (-2,-3) circle (1.5pt);
\fill [color=black] (-2,-4) circle (1.5pt);
\fill [color=black] (-2,-5) circle (1.5pt);
\fill [color=black] (-1,-1) circle (1.5pt);
\fill [color=black] (-1,-2) circle (1.5pt);
\fill [color=black] (-1,-3) circle (1.5pt);
\fill [color=black] (-1,-4) circle (1.5pt);
\fill [color=black] (-1,-5) circle (1.5pt);
\fill [color=black] (0,-1) circle (1.5pt);
\fill [color=black] (0,-2) circle (1.5pt);
\draw  (-3,-4) node {$\diamondsuit$};
\draw  (-3,-5) node {$\diamondsuit$};
\fill [color=black] (7,-1) circle (1.5pt);
\fill [color=black] (7,-2) circle (1.5pt);
\draw  (7,-4) node {$\diamondsuit$};
\draw  (7,-5) node {$\diamondsuit$};
\draw (8,-1) node {$\bowtie$};
\draw (8,-2) node {$\bowtie$};
\draw (8,-3) node {$\bowtie$};
\draw (8,-4) node {$\bowtie$};
\draw (8,-5) node {$\bowtie$};
\draw (16,-5) node {$\bowtie$};
\draw (16,-3) node {$\bowtie$};
\draw (16,-4) node {$\bowtie$};
\draw (16,-2) node {$\bowtie$};
\draw (16,-1) node {$\bowtie$};
\end{scriptsize}
\end{tikzpicture}
\end{center}

Such a morphism $q$ exists by construction, and there are only a finite number of possibilities for $q$. Indeed, let $x$ and $y$ be the branch points of $p$. A suitable $q$ can be described by the immersion of $E$ in $\mathbb{P}^2$ given by the linear series of the fibers, followed by the projection from a point not belonging to the image of $E$.  Assume that we have such an immersion. The point with ramification index three is an inflection of the curve in $\mathbb{P}^2$ and the projection point lies over the tangent in this point. Moreover, the lines linking $x$ and $y$ with the projection point are tangent to the curve in certain points $x'$ and $y'$ respectively.   

Assume by contradiction that there is a positive dimensional family of possible $q$'s. Given a particular immersion determined by one such $q$, the projection point is determined by the intersection of the tangent to an inflexion point and the lines through $x$ and $y$ tangent to the curve. If we move $x$ and $y$ by a point $z\in E\subset \mathbb{P}^2$, we change the immersion of the curve, but we keep the same planar equation. If there is a one dimensional family of suitable morphisms $q$, then the point where the new tangents through $x+z$ and $y+z$ intersect should be over the tangent to the inflexion point, giving another morphism $q$ in the family. Doing the effective computations we find that for a general $z$ it does not happen, and hence, there are only a finite number of suitable $q$'s. $\Diamond$

(D6.10) $\boldsymbol{t=4}$, $\boldsymbol{r=0}$, $\boldsymbol{p=0}$, $\boldsymbol{s=8}$, $\boldsymbol{k=0}$. 

The generating vector $\left(i, i, j,j,j,j\right)$ satisfies all conditions.

\textbf{Claim:} The map $\eta$ has $1$-dimensional fibers.

\emph{Proof of the claim:} Given a curve $[C]$ in the image of $\mathcal{D}$, we recover $\eta^{-1}([C])$ taking the image of the branch points of its bielliptic morphism by a suitable $q$ together with the other four branch points of $q$. By a suitable $q$ we mean that the branch points of $p$  are non ramified lying over a branch point, and $q$ has generic ramification.

In the following figure we show the ramification of the morphisms $p$ and $q$. 
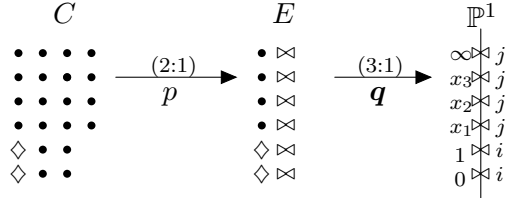
\begin{center}
\begin{tikzpicture}[line cap=round,line join=round,>=triangle 45,x=0.32cm,y=0.32cm]
\clip(-4,-7) rectangle (19,1.5);
\draw (16,0)-- (16,-7);
\draw (15,1.5) node[anchor=north west] {$\mathbb{P}^1$};
\draw (7,1.5) node[anchor=north west] {$E$};
\draw (-2,1.5) node[anchor=north west] {$C$};
\draw [->,color=black] (1,-2) -- (6,-2);
\draw [->] (10,-2) -- (14,-2);
\draw (2,-0.7) node[anchor=north west] {$_{(2:1)}$};
\draw (10.5,-0.7) node[anchor=north west] {$_{(3:1)}$};
\draw (11,-2) node[anchor=north west] {$\boldsymbol{q}$};
\draw (2.5,-2) node[anchor=north west] {$p$};
\draw (16.2,-5.3) node[anchor=north west] {$_{i}$};
\draw (16.2,-4.3) node[anchor=north west] {$_{i}$};
\draw (16.2,-3.3) node[anchor=north west] {$_{j}$};
\draw (16.2,-2.3) node[anchor=north west] {$_{j}$};
\draw (16.2,-1.3) node[anchor=north west] {$_{j}$};
\draw (16.2,-0.3) node[anchor=north west] {$_{j}$};
\draw (14.5,-5.5) node[anchor=north west] {$_{0}$};
\draw (14.5,-4.5) node[anchor=north west] {$_{1}$};
\draw (14.3,-3.5) node[anchor=north west] {$_{x_1}$};
\draw (14.3,-2.5) node[anchor=north west] {$_{x_2}$};
\draw (14.3,-1.5) node[anchor=north west] {$_{x_3}$};
\draw (14.3,-0.5) node[anchor=north west] {$_{\infty}$};
\begin{scriptsize}
\fill [color=black] (-3,-1) circle (1.5pt);
\fill [color=black] (-3,-2) circle (1.5pt);
\fill [color=black] (-3,-3) circle (1.5pt);
\fill [color=black] (-3,-4) circle (1.5pt);
\fill [color=black] (-2,-1) circle (1.5pt);
\fill [color=black] (-2,-2) circle (1.5pt);
\fill [color=black] (-2,-3) circle (1.5pt);
\fill [color=black] (-2,-4) circle (1.5pt);
\fill [color=black] (-2,-5) circle (1.5pt);
\fill [color=black] (-2,-6) circle (1.5pt);
\fill [color=black] (-1,-1) circle (1.5pt);
\fill [color=black] (-1,-2) circle (1.5pt);
\fill [color=black] (-1,-3) circle (1.5pt);
\fill [color=black] (-1,-4) circle (1.5pt);
\fill [color=black] (-1,-5) circle (1.5pt);
\fill [color=black] (-1,-6) circle (1.5pt);
\fill [color=black] (-1,-2) circle (1.5pt);
\fill [color=black] (0,-1) circle (1.5pt);
\fill [color=black] (0,-2) circle (1.5pt);
\fill [color=black] (0,-3) circle (1.5pt);
\fill [color=black] (0,-4) circle (1.5pt);
\draw  (-3,-6) node {$\diamondsuit$};
\draw  (-3,-5) node {$\diamondsuit$};
\fill [color=black] (7,-1) circle (1.5pt);
\fill [color=black] (7,-2) circle (1.5pt);
\fill [color=black] (7,-3) circle (1.5pt);
\fill [color=black] (7,-4) circle (1.5pt);
\draw  (7,-6) node {$\diamondsuit$};
\draw  (7,-5) node {$\diamondsuit$};
\draw (8,-1) node {$\bowtie$};
\draw (8,-2) node {$\bowtie$};
\draw (8,-3) node {$\bowtie$};
\draw (8,-4) node {$\bowtie$};
\draw (8,-5) node {$\bowtie$};
\draw (8,-6) node {$\bowtie$};
\draw (16,-6) node {$\bowtie$};
\draw (16,-5) node {$\bowtie$};
\draw (16,-3) node {$\bowtie$};
\draw (16,-4) node {$\bowtie$};
\draw (16,-2) node {$\bowtie$};
\draw (16,-1) node {$\bowtie$};
\end{scriptsize}
\end{tikzpicture}
\end{center}

Such a morphism $q$ exists by construction, and we claim that there is a one dimensional family of possibilities for $q$. 

Indeed, the elliptic curve $E$ is given, with two marked points $x$ and $y$, and we are looking for a $q:E\rightarrow \mathbb{P}^1$ with generic ramification and the two marked points over a branch point but non-ramified. 

Each immersion of $E$ in $\mathbb{P}^2$ is given by a line bundle $a\in \mathrm{Pic}^3(E)$. If we consider the projection $\pi:E^{(3)}\rightarrow \mathrm{Pic}^3(E)$, the fibers of this morphism are $\mathbb{P}^2$'s given by the linear series. A morphism to $\mathbb{P}^1$ of order $3$ can be seen as a line in this $\mathbb{P}^2$ with no base point, that is, not contained in any divisor $E_x$.

Given two points $x,y\in E$, for each $\mathbb{P}^2=\pi^{-1}(a)$ we have four points of type $x+2x'$ and four of type $y+2y'$, hence, there are $16$ lines that contain one of each type. In this same fiber of $\pi$ there are $9$ points of type $3Q$. Therefore, for $a$ general, at least one of the $16$ lines through  $x+2x'$ and $y+2y'$ will not contain a point of type $3Q$, and therefore, we deduce that given $x,y\in E$ there is a $1$-dimensional family ($\mathrm{dim}\,\mathrm{Pic}(E)=1$) of morphisms of degree 3 from $E$ to $\mathbb{P}^1$ with generic branching type and $x,y$ non ramified but with image a branch point.  

Hence, we can recover a one dimensional family of $(D, \rho)\in \mathcal{D}$. Since for each $D$ we find a different $B$, there is a one dimensional family of curves $\widetilde{B}\subset C^{(2)}$ for each $[C]\in \eta(\mathcal{D})$.$\Diamond$

\vspace{5pt}

\textbf{Assume finally that $\boldsymbol{l=2}$.} Assume $\langle i,j \rangle=D_4$.  

First, we observe that we can embed $D_4$ in $S_4$ in such a way that $i$ is odd and $j$ is even. Since we will need the product one condition when constructing the generating vector, we need to impose $\frac{t}{2}+\frac{r}{2}$ to be even, or what is the same, $t+r$ to be multiple of four. Moreover, we can also embed $D_4$ in $S_4$ in such a way that $i$ is even and $j$ is odd hence we impose $s+r$ to be multiple of four.

Since in $D_4$ we have five conjugacy classes $[1],\ [i],\ [j],\ [ij]\ \textrm{and}\ [(ij)^2]$, all branch points in $D\rightarrow D/D_4$ will be considered in either $t,s,r$ or $k$, thus the Riemann-Hurwitz formula reads $2h-2=-16+2s+2t+3r+k$.

The condition $g\geq 2$ is equivalent to $s\leq 2h-6$, which with the above expression of $h$ becomes
\begin{equation}\label{condg2}
s\leq -14+2s+2t+3r+k-6 \Leftrightarrow 20\leq s+2t+3r+k.
\end{equation}

We have seen that necessarily $t+r\leq 6$. By our previous conditions and considering that we need a compatible generating vector, we find the following three possibilities:

(D4.1) $\boldsymbol{t=0}$, $\boldsymbol{r=4}$. We need that $s>0$ to obtain a generating vector. Since $s+r$ should be multiple of four, we obtain that $s$ is multiple of four. From \eqref{condg2} we deduce that $8\leq s+k$. 

Depending on the parity of $\frac{k}{4}$ we consider the following generating vector:

\begin{itemize}
\item If $\frac{k}{4}$ is even: $\left(ij, ji, {j,\stackrel{s/2}{\dots}, j} , {(ij)^2,\stackrel{k/4}{ \dots}, (ij)^2}\right)$.

\item If $\frac{k}{4}$ is odd: $\left(ij, ji, {j,\stackrel{s/2-1}{\dots}, j}, j(ij)^2, {(ij)^2, \stackrel{k/4}{\dots}, (ij)^2}\right)$.
\end{itemize}

We note that $\nu(\beta_B)=\frac{k+4}{2}$ and hence, we obtain that $g(B/\langle \beta_B\rangle )=\frac{s}{4}\geq 1$.

(D4.2) $\boldsymbol{t=2}$, $\boldsymbol{r=2}$. Since $s+r$ should be multiple of four, we obtain that $s=4\alpha+2$ with $\alpha\in \mathbb{Z}_{\geq 0}$. From \eqref{condg2} we deduce that $10\leq s+k$. 

Depending on the parity of $\frac{k}{4}$ we consider the following generating vector:

\begin{itemize}
\item If $\frac{k}{4}$ is even: $\left(i, ij, {j,\stackrel{s/2}{\dots}, j} , {(ij)^2, \stackrel{k/4}{\dots}, (ij)^2}\right)$.

\item If $\frac{k}{4}$ is odd: $\left(i, {j,\stackrel{s/2}{\dots}, j}, ij, {(ij)^2, \stackrel{k/4}{\dots}, (ij)^2}\right)$.
\end{itemize}

We note that $\nu(\beta_B)=\frac{k+2+2}{2}$ and hence $g(B/\langle \beta_B\rangle )=\frac{s-2}{4}\geq 0$.

(D4.3) $\boldsymbol{t=4}$, $\boldsymbol{r=0}$. We need that $s>0$ to generate. Since $s+r$ should be multiple of four, we obtain that $s$ is multiple of four. From \eqref{condg2} we deduce that $12\leq s+k$. 

Depending on the parity of $\frac{k}{4}$ we consider the following generating vector:

\begin{itemize}
\item If $\frac{k}{4}$ is even: $\left(i, i, {j,\stackrel{s/2}{\dots}, j} ,{(ij)^2,\stackrel{k/4}{ \dots}, (ij)^2}\right)$.

\item If $\frac{k}{4}$ is odd: $\left(i, i, {j,\stackrel{s/2-1}{\dots}, j}, j(ij)^2, {(ij)^2, \stackrel{k/4}{\dots}, (ij)^2}\right)$.
\end{itemize}

We note that $\nu(\beta_B)=\frac{k+4}{2}$ and hence, we obtain that $g(B/\langle \beta_B\rangle )=\frac{s-4}{4}\geq 0$.

Finally, we remark that for these three families we have that $\widetilde{B}^2=4$, the maximum possibility by the Hodge index theorem. Therefore, $\widetilde{B}$ is algebraically equivalent to two times a coordinate curve. Moreover, since a coordinate curve has positive self-intersection, the restriction of $\mathrm{Pic}^0(C^{(2)})$ to $\mathrm{Pic}^0(C_P)$ is injective. Then, since the restriction of $\widetilde{B}$ to a coordinate curve consists of two points, it is  linearly equivalent to the restriction of the sum of two coordinate curves. Then, the divisor $\widetilde{B}-(C_{P_1}+C_{P_2})$ is on $\mathrm{Pic}^0(C^{(2)})$ and restricted to $\mathrm{Pic}^0(C_P)$ it is zero, therefore, since this morphism is injective, we deduce that $\widetilde{B}$ is linearly equivalent to the sum of two coordinate curves, and hence $h^0(C^{(2)}, \mathcal{O}_{C^{(2)}}(\widetilde{B}))\geq 2$. 

These are all possible pairs $(C, \widetilde{B})$ with $\widetilde{B}\subset C^{(2)}$ with degree two and positive self-intersection. Therefore, we have finished the proof of Theorem \ref{clasdeg2}.

\section{Further comments}\label{sec:furthcomm}

\begin{remark}
Notice that in all cases $\widetilde{B}^2\leq p_a(\widetilde{B})-g(C)+2$, thus satisfying the inequality in Corollary 4.7 of \cite{MPP2} even if in the cases with $g(C)=2$ the surface $C^{(2)}$ is not of general type, and therefore the hypothesis are not fulfilled.
\end{remark}

We can give more information about the curves $D$ in relation to the curves $C$ and $B$.

\begin{proposition}
We have the following isogeny for the Jacobian variety of any of the curves $D$ that appear in Theorem \ref{clasdeg2}:
\[
J_D\approx J_C \times J_B \times J_{D/\langle ij \rangle}.
\]
\end{proposition}

\begin{proof}
We can decompose $D_n=\langle i,j \rangle $, with $i$ and $j$ two involutions, as 
\[
D_n=\langle ij \rangle \cup \langle i \rangle  \cup \langle (ij)i \rangle \cup \langle (ij)^2i \rangle\cup \dots  \cup \langle (ij)^{n-1}i \rangle.
\]

In our case $D/D_n\cong \mathbb{P}^1$, so by \cite{KR} with $t=n+1$ we obtain that
\[
J_D^{n}\approx J_{D/\langle ij \rangle}^{n} \times J_{D/\langle i \rangle}^{2}\times J_{D/\langle (ij)i \rangle}^{2}\times \dots \times J_{D/\langle (ij)^{n-1}i \rangle}^{2} .
\]

Since $i$ is a conjugate of all $(ij)^{2k}i$ and $j$ is a conjugate of all $(ij)^{2k+1}i$ we deduce that 
\[
\begin{array}{cc}
J_{D/\langle i \rangle} \cong J_{D/\langle (ij)^{2k}i \rangle}\cong J_B & J_{D/\langle j \rangle} \cong J_{D/\langle (ij)^{2k+1}i \rangle}\cong J_C.
\end{array}
\] 
Therefore,
\[
J_D^{n}\approx J_{D/\langle ij \rangle}^{n} \times J_{B}^{n}\times J_{C}^{n}
\]
and applying Poincar\'e duality we obtain the stated isogeny.
\end{proof}

\subsection{Degree $d$ curves in $C^{(2)}$ with arithmetic genus in the Brill-Noether range}

Now, we consider again the question about curves with arithmetic genus in the Brill-Noether range and positive self-intersection. Let $(C, \widetilde{B})$ be a pair of curves with $C$ smooth, such that $\widetilde{B}\subset C^{(2)}$ with degree $d$,  
\begin{equation}\label{cond}
q(C^{(2)})=g(C)<p_a(\widetilde{B})<2g(C)-1
\end{equation}
and $\widetilde{B}$ has positive self-intersection.

By the Castelnuovo-Severi inequality (\cite{Acc2}), for a diagram of curves as in Theorem \ref{charcurves} the following inequality is satisfied:
\begin{equation}\label{CSI}
g(D)\leq 2g(B)+dg(C)+d-1.
\end{equation}
Thus, we have a necessary condition for given curves $C, B$ and $D$ to lay in a diagram as in Theorem \ref{charcurves}. 

Moreover, when $d$ is a prime number, by \cite[Theorem 3.2]{Acc} we have a criteria to decide if a curve completing the diagram can exist. That is, a curve $F$ completing the diagram must satisfy 
\[
g(D)+2dg(F)\leq 2g(B)+dg(C)+d-1.
\]
When this inequality is not satisfied, the diagram does not complete and hence by Theorem \ref{charcurves}, there is a degree one map $B\rightarrow C^{(2)}$ with image $\widetilde{B}$.

By Lemma \ref{selfint} and inequality \eqref{CSI} we get that
\[
\begin{array}{c}
g(D)=\widetilde{B}^2+1+2d(g(C)-1)-(p_a(\widetilde{D})-g(D))\leq 2g(B)+dg(C)+d-1 \\[2mm]
\Leftrightarrow \widetilde{B}^2\leq 2g(B)+d(3-g(C))-2+(p_a(\widetilde{D})-g(D)).
\end{array}
\]

By \eqref{cond}, necessarily $g(B)\leq p_a(\widetilde{B})\leq 2g(C)-2$, so we obtain that for $g(C)\geq 4$ 
\[
d\leq \frac{4g(C)-6-\widetilde{B}^2+(p_a(\widetilde{D})-g(D))}{g(C)-3}.
\]

Hence, for a fixed $g(C)$, we have a relation between the self-intersection of $\widetilde{B}$, its degree and the singularities of $\widetilde{D}$. 

Assume that $\widetilde{B}$ is smooth. Then, we have a diagram of curves
\[
\xymatrix{
\widetilde{D} \ar[r]^{(2:1)} \ar[d]^{(d:1)}& B \\
C &
}
\]
that induces a map $p:\widetilde{D}\rightarrow C\times B$. Since $2$ is a prime number, $p$ is birational into its image. Hence, $p_a(\widetilde{D})\leq p_a(i(\widetilde{D}))$.

By adjunction in $C\times B$ and the Castelnuovo-Severi inequality we get that 
\[
p_a(i(\widetilde{D}))\leq 2g(B)+dg(C)+d-1.
\]
Hence, by Lemma \ref{selfint} and inequality \eqref{cond}, we obtain that for $g(C)\geq 4$
\[
d\leq \dfrac{4g(C)-7}{g(C)-3}.
\]
Thus, for $g(C)\geq 9$ a curve with arithmetic genus in the Brill-Noether range and positive self-intersection should have degree at most $4$. This inequality motivates the study of curves in $C^{(2)}$ with low degree.

By Theorem \ref{clasdeg2} we obtain a negative answer for the existence of such curves with degree $2$.

\begin{corollary}
There are no pairs of curves $(C, \widetilde{B})$ with $C$ smooth and $\widetilde{B}\subset C^{(2)}$ with degree two, $g(C)<p_a(\widetilde{B})<2g(C)-1$ and $\widetilde{B}^2>0$.
\end{corollary}

\subsection{Curves in $C\times C$ with arithmetic genus in the Brill-Noether range}

We remark finally that some of the curves $\widetilde{D}\subset C\times C$ have, in fact, arithmetic genus in the Brill-Noether range and positive self-intersection. Indeed, since by definition $\widetilde{D}$ is the preimage of $\widetilde{B}$ by $\pi_C$, we have
\[
\widetilde{D}^2=(\pi_C^*\widetilde{B})^2=\pi_C^*(\widetilde{B}^2)={\rm deg}\,\pi_C\ \widetilde{B}^2=2\widetilde{B}^2
\]
and moreover 
\[
q(C\times C)=2g\qquad  p_a(\widetilde{D})=g(D)+\frac{1}{2}(r+k).
\]

Therefore, if the curve $\widetilde{B}$ has positive self-intersection, so does $\widetilde{D}$. Moreover, if $2g< h+ \frac{1}{2}(r+k) \leq 4g-2$, then $\widetilde{D}$ has arithmetic genus in the Brill-Noether range and positive self-intersection.

We look at Theorem \ref{clasdeg2} and note that these inequalities are satisfied in all cases defined by the action of $D_{10}$, the four cases defined by the action of $D_6$ with $g(C)=3$ and all the cases defined by the action of $D_4$.

Therefore, we have found some examples of curves with arithmetic genus in the Brill-Noether range and positive self-intersection in an irregular surface. We note that all our examples, as well as those found in \cite{PiAl}, satisfy $p_a(\widetilde{D})=2q(S)-2$. That is, the curves are on the boundary of the Brill-Noether range.

\vskip 15pt

\textbf{Acknowledgements}
The author thanks Rita Pardini for all the time and useful suggestions given during the author's stay in the University of Pisa and afterwards. Many thanks also to Gian Pietro Pirola for suggesting us to look into $C\times C$ for curves with arithmetic genus in the Brill-Noether range. The most sincere gratitude to Miguel Angel Barja and Joan Carles Naranjo for the multiple discussions and the huge amount of time devoted to the development of this article. To the anonymous referee for multiple suggestions that greatly improved the exposition. And finally to the Universitat de Barcelona for the research grant and their hospitality afterwards.

The author has been partially supported by the Proyecto de Investigaci\'on MTM2012-38122-C03-02.

\bibliographystyle{plain}
\bibliography{Saez_bibliography}

\end{document}